\documentclass[11pt,reqno]{amsart}

\usepackage{amsmath, amsthm, amssymb}
\DeclareMathOperator{\tr}{tr}
\usepackage{amsfonts}
\usepackage{bm,mathrsfs}
\usepackage{enumitem}

\usepackage[ansinew]{inputenc}
\usepackage{epsfig}
\usepackage{graphicx}
\usepackage[english]{babel}
\usepackage{mathrsfs}
\usepackage{mathtools}
\DeclarePairedDelimiter{\norm}{\lVert}{\rVert}

\usepackage{thmtools}
\theoremstyle{plain}
\declaretheorem[title=Theorem, parent=section]{theorem}
\declaretheorem[title=Lemma,sibling=theorem]{lemma}
\declaretheorem[title=Proposition,sibling=theorem]{proposition}
\declaretheorem[title=Corollary,sibling=theorem]{corollary}

\theoremstyle{definition}
\declaretheorem[title=Definition,sibling=theorem]{definition}
\declaretheorem[title=Remark,sibling=theorem]{remark}
\declaretheorem[title=Remark, numbered=no]{remark*}

\declaretheorem[title=Assumption, numbered=no]{assumption*}

\numberwithin{equation}{section}

\usepackage[backgroundcolor=white, bordercolor=blue,
linecolor=blue]{todonotes}

\usepackage{dsfont}
\usepackage{bbm}

\newcommand{\R}{\mathbb{R}}

\DeclareMathOperator{\Id}{Id}

\DeclareMathOperator*{\osc}{osc}

\newcommand{\average}{{\mathchoice {\kern1ex\vcenter{\hrule height.4pt
width 6pt depth0pt} \kern-9.7pt} {\kern1ex\vcenter{\hrule
height.4pt width 4.3pt depth0pt} \kern-7pt} {} {} }}

\usepackage[hypertexnames=false]{hyperref}

\begin{document}
\allowdisplaybreaks
\title[Asymptotic gradient regularity]{Asymptotic gradient regularity for degenerate or singular fully nonlinear parabolic equations}

\author{Jiangwen Wang}

\address{School of Mathematics and Shing-Tung Yau Center of Southeast University, Southeast University, Nanjing 211189, P.R. China}
\email{\url{jiangwen\_wang@seu.edu.cn}}


\date{\today}



\allowdisplaybreaks

\begin{abstract}
This is a continuation of our recent work~\cite{WZJ26} on asymptotic gradient regularity for nonlinear parabolic equations. In this paper, we consider bounded viscosity solutions of the degenerate or singular fully nonlinear parabolic equation
\begin{equation*}
   u_t-|Du|^\gamma F(D^2u)=f
 \quad    \text{in } Q_1,
\end{equation*}
where $ \gamma>-1 $, $ F $ is uniformly $ (\lambda, \Lambda)$-elliptic with $0<\lambda\leq\Lambda$, $F(0)=0$, and $f\in C^0(Q_1)\cap L^\infty(Q_1)$.
Given any prescribed $\alpha\in(0,1)$, we prove that there exists
$ \varepsilon=\varepsilon(n,\alpha)>0 $ such that
\[
 |\gamma|+\left(\frac{\Lambda}{\lambda}-1\right) \leq \varepsilon
 \quad   \Longrightarrow     \quad
 u  \in C_{\mathrm{loc}}^{1+\alpha,\frac{1+\alpha}{2}}(Q_1).
\]
In particular, the spatial gradient H\"older exponent can be chosen arbitrarily close to $1$ as $ (\gamma, \frac{\Lambda}{\lambda}) \to (0,1) $. This provides a fully nonlinear counterpart to the asymptotic gradient regularity result established in~\cite{WZJ26}.
\end{abstract}

\allowdisplaybreaks

\maketitle

\section{Introduction}
In this paper, we consider asymptotic gradient regularity for the degenerate or singular fully nonlinear parabolic equation
\begin{equation}\label{Intro:maineq1}
  u_t - |Du|^{\gamma} F(D^2u) =f   \quad  \text{in } Q_{1}:= B_{1} \times (-1, 0],
\end{equation}
where $ -1 < \gamma < \infty  $ and $f\in C^{0}(Q_1)\cap L^\infty(Q_1) $. The operator $ F : \mathrm{Sym}(n) \to \mathbb{R} $ satisfies the following structural condition:

\vspace{1mm}

\noindent
\phantomsection\label{F1} {\bf (F1).} $F$ is uniformly elliptic with $F(0) = 0$; namely, there exist constants $0 < \lambda \leq \Lambda$ such that, for all $\mathrm{M}, \mathrm{N} \in \mathrm{Sym}(n)$,
\[
  \mathscr{P}^{-}_{\lambda,\Lambda}(\mathrm{M}-\mathrm{N}) \leq F(\mathrm{M}) - F(\mathrm{N})
  \leq \mathscr{P}^{+}_{\lambda,\Lambda}(\mathrm{M}-\mathrm{N}),
\]
where $\mathscr{P}^{\pm}$ denote the \textit{Pucci extremal operators} \cite{CC95}, i.e.,
\[
  \mathscr{P}^{+}_{\lambda,\Lambda}(\mathrm{X})
  := \lambda \sum_{e_{i}<0} e_{i}(\mathrm{X}) + \Lambda \sum_{e_{i}>0} e_{i}(\mathrm{X}),
  \quad \text{and}  \quad
  \mathscr{P}^{-}_{\lambda,\Lambda}(\mathrm{X})
  := \lambda \sum_{e_{i}>0} e_{i}(\mathrm{X}) + \Lambda \sum_{e_{i}<0} e_{i}(\mathrm{X}).
\]

When $\gamma=0$, equation~\eqref{Intro:maineq1} is uniformly
parabolic; see the classical works of Wang~\cite{W92a,W92b,W92c}
for the corresponding regularity theory. For $\gamma>0$, the diffusion degenerates at points where the spatial gradient vanishes, whereas for $-1<\gamma<0$ it is singular there. The regularity problem therefore involves both the ellipticity of the Hessian operator and the behavior of the gradient-dependent factor near the critical set. Our aim is not merely to obtain a positive H\"older exponent for $ Du $, but to determine whether this exponent can be prescribed arbitrarily close to one as $ \gamma \to 0$ and $\frac{\Lambda}{\lambda}  \to  1$.

For equation~\eqref{Intro:maineq1}, Lee-Lee-Yun
\cite{LLY24} established interior $C^{1,\alpha}$-regularity for some
positive exponent when $\gamma>-1$, assuming that $F$ is convex or
concave and belongs to $C^{1,\kappa}$ for some $\kappa\in(0,1]$.
Their argument proceeds through classical solutions of regularized
Cauchy--Dirichlet problems. A central ingredient is a
Bernstein-type gradient estimate formulated in terms of difference
quotients, which makes it possible to establish the required
solvability without assuming third-order differentiability of the
solutions. This work provides a regularity theory over the full range
$\gamma>-1$, without a smallness restriction on the ellipticity
aperture.

Lee-Lian-Yun-Zhang~\cite{LLYZ25} subsequently developed
an approach based on local boundedness of the time
derivative. They proved this estimate for the homogeneous parabolic
$p$-Laplace equation with $p>2$ and extended their method to fully nonlinear degenerate
equations. For
\begin{equation*}
    u_t- |Du|^\gamma F(D^2u) = 0,
    \quad \gamma>0,
\end{equation*}
with convex or concave $F$, they obtained the optimal spatial gradient
exponent $\frac{1}{1+\gamma}$, Lipschitz continuity of $u$ in time, and temporal H\"older exponent $\frac{1}{2+\gamma} $ for $Du$.
The time-derivative estimate allows the spatial regularity problem to
be reduced to the corresponding elliptic equation on time slices.
Importantly, their fully nonlinear result also removes the
$C^{1,\kappa}$ assumption on $F$ from \cite{LLY24}.
These conclusions concern the homogeneous degenerate problem under
convexity or concavity and are distinct from the perturbative,
nonconvex setting considered here. In a related direction, da Silva-Jiang-Wang~\cite{dJW25} obtained sharp free boundary growth estimates for nonnegative, time-nondecreasing solutions of fully nonlinear degenerate parabolic equations with strong absorption, assuming convexity or concavity of $ F $.

A complementary line of research concerns improved regularity near
the heat equation. Andrade and Santos~\cite{AS22} studied the
inhomogeneous parabolic normalized $p$-Laplace equation
\begin{equation*}
    u_t-\Delta_p^N u=f   \quad   \text{in}  \quad  Q_{1}
\end{equation*}
where
\begin{equation*}
  \Delta_p^N u
    :=\Delta u+(p-2)
    \left\langle D^2u\frac{Du}{|Du|},\frac{Du}{|Du|}\right\rangle.
\end{equation*}
They proved that, for each prescribed
$\alpha\in(0,1)$, bounded viscosity solutions belong locally to
$C^{1+\alpha,\frac{1+\alpha}{2}}$ whenever $p$ is sufficiently close to $ 2$. Their proof transfers the regularity of caloric functions through
compactness and an iteration of affine approximations. In
\cite[Remark~1.1]{AS22}, they asked whether the same improvement
holds for the weighted equation
$u_t-|Du|^\gamma\Delta_p^N u=f$ under the joint smallness condition
$|p-2|+|\gamma|\ll1$.

Lee and Yun~\cite{LY24} established a related perturbative result at
the level of Pucci solution classes. They proved interior and
boundary $C^{1,\alpha}$-estimates for any prescribed
$\alpha\in(0,1)$ when the ellipticity ratio is sufficiently close to
one, and applied these estimates to normalized $p$-Laplace equations.
An important feature of their formulation is that the solution class
is preserved under subtraction of spatial affine functions.
Consequently, the gradient dependence of the normalized operator can
be handled within a fixed uniformly parabolic class throughout the
iteration. For $\gamma=0$, their interior regularity already gives the
conclusion sought here for~\eqref{Intro:maineq1}.

In our recent work~\cite{WZJ26}, we considered the weighted
normalized equation and proved that
\begin{equation*}
    |p-2|+|\gamma|\leq\varepsilon(n,\alpha)
    \quad   \Longrightarrow   \quad
    u\in C_{\mathrm{loc}}^{1+\alpha,\frac{1+\alpha}{2}}(Q_1).
\end{equation*}
This answers the question in \cite[Remark~1.1]{AS22}
in a joint neighborhood of $(p,\gamma)=(2,0)$,
including both positive and negative $\gamma$. The proof combines heat
approximation with an intrinsic improvement of flatness and a
separation into critical and noncritical regimes. The present paper
provides a fully nonlinear counterpart of that result, with the
smallness of $|p-2|$ replaced by a small ellipticity aperture for the
Hessian operator. The two operator classes are not identical: the
normalized $p$-Laplacian has an additional dependence on the gradient
direction, whereas here $F$ may depend nonlinearly on the Hessian.

The preceding results lead to the following question:

\vspace{2mm}

\textit{`` Does the small-aperture improvement of~\cite{LY24} persist in the presence of $|Du|^\gamma$, without convexity or differentiability assumptions on $F$?  ''}

\vspace{2mm}

A direct application of the fixed Pucci-class estimates is not
available. Even when $\frac{\Lambda}{\lambda}$ is close to $ 1$, the factor $|Du|^\gamma$ may vanish or become unbounded near $Du=0$.
Thus small ellipticity aperture does not, by itself, restore fixed
positive parabolicity constants. Our main theorem gives an affirmative
answer when the ellipticity aperture and $|\gamma|$ are simultaneously small.

\begin{theorem}\label{Main:thm}
  Suppose $ F $ satisfies \hyperref[F1]{\bf (F1)}. Let $u\in C^{0}(Q_1)$ be a bounded viscosity solution of \eqref{Intro:maineq1}, where $ -1<\gamma<\infty $, and $f\in C^{0}(Q_1)\cap L^\infty(Q_1) $. Given $\alpha\in(0,1)$, there exists
$
 \varepsilon =\varepsilon(\alpha, n)>0 $ such that if
\[
|\gamma|+ \left(\frac{\Lambda}{\lambda}-1\right) \leq \varepsilon,
\]
then $ u\in C^{1+\alpha, \frac{1+\alpha}{2}}(Q_{\frac{1}{2}}) $. Moreover,
\[
 ||u||_{C^{1+\alpha,\frac{1+\alpha}{2}}(Q_{\frac{1}{2}})}
 \le \mathrm{C},
\]
where $ \mathrm{C} $ depends only on $ n $, $ \alpha $, $ \lambda $,  $ ||u||_{L^\infty(Q_1)} $, and $ ||f||_{L^\infty(Q_1)} $.
\end{theorem}

\begin{remark}
Compared with~\cite{LLY24,LLYZ25}, Theorem \ref{Main:thm} dispenses with convexity or concavity within a perturbative regime and treats
bounded continuous forcing together with both signs of $\gamma$.
It does not improve the optimal fixed-$\gamma$ estimates
of~\cite{LLYZ25} in their homogeneous convex or concave setting.
\end{remark}

Taking $\gamma=0$ in Theorem~\ref{Main:thm}, we obtain the
following special case of the interior estimate in
\cite[Theorem~1.1]{LY24} for translation-invariant fully
nonlinear equations.    

\begin{corollary}[Uniformly parabolic case]
\label{cor:uniformly-parabolic}   
Let $\alpha\in(0,1)$, and suppose that $F$ satisfies
\hyperref[F1]{\bf (F1)}. Let $u\in C^0(Q_1)$ be a
bounded viscosity solution of
\begin{equation*}
    u_t-F(D^2u)=f
    \quad\text{in }Q_1,
\end{equation*}
where $f\in C^0(Q_1)\cap L^\infty(Q_1)$.
There exists $\varepsilon=\varepsilon(n,\alpha)>0$ such
that, if
\begin{equation*}
    \frac{\Lambda}{\lambda}-1\leq\varepsilon,
\end{equation*}
then $u\in C^{1+\alpha,\frac{1+\alpha}{2}}
(Q_{\frac{1}{2}})$ and
\begin{equation*}
    \|u\|_{C^{1+\alpha,\frac{1+\alpha}{2}}
    (Q_{\frac{1}{2}})}
    \leq \mathrm{C} \left(
        \|u\|_{L^\infty(Q_1)}
        +\|f\|_{L^\infty(Q_1)}
    \right),
\end{equation*}
where $\mathrm{C} $ depends only on $n$, $\alpha$, and $\lambda$.
\end{corollary}

\begin{proof}
Set
\begin{equation*}
    A:=\|u\|_{L^\infty(Q_1)}
       +\|f\|_{L^\infty(Q_1)}.
\end{equation*}
The case $A=0$ is immediate. For $A>0$, define
\begin{equation*}
    v:=\frac{u}{A},
    \qquad
    F_A(\mathrm{M}):=\frac{F(A\mathrm{M})}{A},
    \qquad
    g:=\frac{f}{A}.
\end{equation*}
Then $F_A(0)=0$, $F_A$ has the same ellipticity constants
as $F$, and $v$ solves
\begin{equation*}
    v_t-F_A(D^2v)=g
    \quad\text{in }Q_1,
\end{equation*}
with $\|v\|_{L^\infty(Q_1)}\leq1$ and
$\|g\|_{L^\infty(Q_1)}\leq1$.
Applying Theorem~\ref{Main:thm} with $\gamma=0$ yields
\begin{equation*}
    \|v\|_{C^{1+\alpha,\frac{1+\alpha}{2}}
    (Q_{\frac{1}{2}})}
    \leq \mathrm{C}(n,\alpha,\lambda).
\end{equation*}
Multiplying by $A$ proves the estimate.    
\end{proof}

\subsection{Outline of the proof}

We follow the critical/noncritical strategy of \cite{WZJ26}.
The main additional difficulty is to establish the required estimates
for a general uniformly elliptic operator $ F $, without relying on the explicit coefficient structure of the normalized $ p $-Laplacian.

\medskip
\noindent\emph{Step 1. Uniform compactness and heat approximation.}
For the shifted equation
\begin{equation*}
  u_t-|q+Du|^\gamma F(D^2u)=f,
\end{equation*}
temporal regularization, spatial doubling of variables, and quartic
barriers yield uniform equicontinuity for bounded shifts $q$.
The delicate point is that $|\xi|^\gamma\to1$ as $\gamma\to0$
is not uniform near $\xi=0$, where $\xi$ denotes the total test gradient.
Thus standard viscosity stability does not immediately identify the
limiting equation at zero-total-gradient contacts. We address these
contacts through quartic-contact inequalities and a separate
limit-identification argument, obtaining heat approximation when
$|\gamma|+\left(\frac{\Lambda}{\lambda}-1 \right)$ and the source $ f$ are sufficiently small.

\medskip
\noindent\emph{Step 2. Intrinsic flatness improvement and stopping.}
Choose $\beta= \frac{1+\alpha}{2} >\alpha$. Heat approximation improves affine
flatness on the intrinsic cylinders
$B_r\times(-r^{2-\beta\gamma},0]$.
Although $F$ need not be homogeneous, its rescalings preserve uniform
ellipticity and closeness to the trace operator. The rescaled source
also remains small because $1-\beta(1+\gamma)>0$.
We iterate while the affine slopes satisfy
$|l_k|\leq3r_k^\beta$.
Either this condition holds at every scale, yielding an expansion at
a critical point, or the first stopping scale produces a normalized
slope in a fixed annulus.

\medskip
\noindent\emph{Step 3. Noncritical Lipschitz estimate and reduction.}
A large affine slope alone does not exclude cancellation by $ Du $.
We therefore first prove a spatial $ 1 $-Lipschitz estimate for the
normalized deviation. In place of the explicit matrix-coefficient
calculations in \cite{WZJ26}, we control $F(\mathrm{X})-F(\mathrm{Y})$ by Pucci inequalities and absorb the gradient-weight mismatch using a concave penalization and the small oscillation. Consequently, every total test gradient stays in a fixed annulus. Clamping the gradient factor outside this annulus gives a globally uniformly parabolic perturbation of the heat equation.
This is where the fixed Pucci-class approach of \cite{LY24} becomes
applicable; it cannot be used directly before the degeneracy or
singularity has been excluded.

\medskip
\noindent\emph{Step 4. Recovery of space--time regularity.}
We combine the critical and noncritical estimates on common cylinders
with time length $r^{\vartheta_\gamma}$, where
$\vartheta_\gamma=\max\{2,2-\beta\gamma\}$.
Comparing affine expansions at nearby centers yields the spatial and
temporal estimates. The choice $\beta>\alpha$ compensates for the
loss in temporal exponents when $\gamma<0$: for sufficiently small
$|\gamma|$,
$ \frac{\beta}{\vartheta_\gamma}  \geq  \frac{\alpha}{2}$ and
$ \frac{1+\beta}{\vartheta_\gamma} \geq \frac{1+\alpha}{2} $.
Undoing the normalization then proves Theorem~\ref{Main:thm}.

\subsection{Structure of the paper}
The paper is organized as follows. Section~\ref{Section 2} introduces the viscosity framework and records the normalization and intrinsic scaling properties. Section~\ref{Section 3} establishes uniform compactness and heat approximation for the shifted equation. These results are used in Section~\ref{Section 4} to prove the one-step intrinsic improvement of flatness. Section~\ref{Section 5} develops the critical iteration and analyzes the first stopping scale, where a noncritical Lipschitz estimate allows us to reduce the equation to a uniformly parabolic perturbation of the heat equation. Finally, Section~\ref{Section 6} combines the critical and noncritical estimates into a common affine expansion and completes the proof of Theorem~\ref{Main:thm} by undoing the normalization and comparing affine expansions at nearby points to obtain the required spatial and temporal regularity.

\subsection{Notations}
We summarize below the basic notation used throughout the paper.
\begin{itemize}

\item For $z_0=(x_0,t_0)\in\R^{n+1}$ and $r>0$, set
\[
 Q_r(z_0):=B_r(x_0)\times(t_0-r^2,t_0],
 \qquad Q_r:=Q_r(0,0).
\]

\item For $ \mu >0$ we also use the intrinsic cylinder
\begin{equation*}
 Q_r^\mu(z_0)
 :=B_r(x_0)\times(t_0-r^2\mu^{-\gamma},t_0],
 \qquad Q_r^\mu:=Q_r^\mu(0,0).
\end{equation*}

\item $\mathrm{Sym}(n)$ denotes the space of all real symmetric $n\times n$ matrices.

\item For $0<\sigma<1$ and a parabolic cylinder $Q$, define
\[
 [v]_{C^{\sigma,\frac{\sigma}{2}}(Q)}
 :=\sup_{(x,t)\ne(y,s)\in Q}
 \frac{|v(x,t)-v(y,s)|}{|x-y|^\sigma+|t-s|^{\frac{\sigma}{2}}}.
\]
We use the standard norm
\begin{align*}
 \norm{u}_{C^{1+\alpha,\frac{1+\alpha}{2}}(Q)}
 :=&\ \norm{u}_{L^\infty(Q)}+\norm{Du}_{L^\infty(Q)}
 +[Du]_{C^{\alpha,\frac{\alpha}{2}}(Q)}\\
 &+\sup_{(x,t), (x,s) \in Q, t\ne s}
 \frac{|u(x,t)-u(x,s)|}{|t-s|^{\frac{1+\alpha}{2}}},
\end{align*}
with the suprema restricted to points for which the relevant pairs belong to $ Q $.

\item $\Id_n$ denotes the $ n \times n $ identity matrix.

\item For a matrix $\mathrm{M} $, $\|\mathrm{M}\|$ denotes its operator norm.

\item For nonnegative quantities $A$ and $B$, we write $A\asymp B$
if $\mathrm{C}^{-1}B \leq A \leq \mathrm{C}B$ for some constant $\mathrm{C}\geq1$ depending only on the parameters specified in the relevant estimate.
\end{itemize}

{\bf Acknowledgments}. The author is deeply grateful to Dr Yini Zhang (University of Bristol, UK) for helpful comments on an earlier version of this paper. Generative AI tools were used to assist with language editing and improving the presentation of the manuscript. The author takes full responsibility for the content of the manuscript.

\vspace{2mm}

\section{Preliminaries}\label{Section 2}
In this section, we recall the viscosity framework for
\eqref{Intro:maineq1} and record the normalization and intrinsic
scaling properties used throughout the proof.

\subsection{Viscosity solutions}
For $\gamma>0$, the operator in \eqref{Intro:maineq1} is continuous
also at zero gradient, and we use the standard viscosity definition;
see \cite[Definition~3.4]{IS13} or \cite[Section~8]{CIL92}.
When $\gamma=0$, the equation is understood as $u_t-F(D^2u)=f$
and is treated in the same standard sense. For $-1<\gamma<0$,
we use the formulation below; cf.\ \cite[Definition~2.2]{LLY24}.
All viscosity inequalities are imposed at time-interior points.
For solutions continuous up to $t=0$, estimates at the top time are
obtained by passage to the limit from $t<0$.

\begin{definition}\label{Section2:def1}
Let $-1<\gamma<0$, and let $u$ be a locally bounded upper
semicontinuous function in $Q_1$. We say that $u$ is a viscosity
subsolution of \eqref{Intro:maineq1} if both of the following
requirements hold at every $(x_0,t_0)\in B_1\times(-1,0)$:
\begin{enumerate}[label=(\roman*)]
\item For every $C^{2,1}$ test function $\varphi$ touching $u$
from above at $(x_0,t_0)$ with $D\varphi(x_0,t_0)\ne0$, we have
\begin{equation*}
 \varphi_t(x_0,t_0)
 -|D\varphi(x_0,t_0)|^\gamma F(D^2\varphi(x_0,t_0))
 \le f(x_0,t_0).
\end{equation*}
\item For every $\delta_1,\delta_2>0$ and
$\varphi\in C^1(I)$, where $I=(t_0-\delta_1,t_0+\delta_1)$ and
$B_{\delta_2}(x_0)\times I\subset B_1\times(-1,0)$, if
\begin{equation*}
 u(x_0,t_0)-\varphi(t_0)\ge u(x_0,t)-\varphi(t)
 \quad\text{for all }t\in I,
\end{equation*}
and
\begin{equation*}
 \sup_{t\in I}\bigl[u(x,t)-\varphi(t)\bigr]
 =u(x_0,t_0)-\varphi(t_0)
 \quad\text{for all }x\in B_{\delta_2}(x_0),
\end{equation*}
then $\varphi'(t_0)\le f(x_0,t_0)$.
\end{enumerate}
A locally bounded lower semicontinuous function $u$ is a viscosity
supersolution if $-u$ is a viscosity subsolution with operator
$F^\sharp(\mathrm{M}):=-F(-\mathrm{M})$ and right-hand side $-f$.
Equivalently, the supersolution conditions use lower tests, reverse
the inequalities, and replace the temporal supremum by an infimum.
A continuous function is a viscosity solution if it is both a
viscosity subsolution and a viscosity supersolution.
\end{definition}

Throughout, right-hand sides are bounded and continuous on the
cylinders under consideration. For a fixed $q\in\mathbb R^n$,
the viscosity interpretation of
$w_t-|q+Dw|^\gamma F(D^2w)=f$ is inherited from
$v(x,t)=w(x,t)+q\cdot x$. In particular, a test $\varphi$ for $w$
is critical precisely when $q+D\varphi=0$.

We will use the following consequence at the centers of quartic time
barriers and in the heat-approximation argument.
\begin{lemma}\label{Section2:lemma1}
Assume that $ F $ satisfies \hyperref[F1]{\bf (F1)}, $f\in C^{0}(Q_1)\cap L^\infty(Q_1) $ and $ \gamma>-\frac{2}{3} $. Let $u$ be a viscosity subsolution of \eqref{Intro:maineq1} with the conventions above. If
\begin{equation*}
  \psi(t) + A|x-x_0|^4,\quad A>0,  \quad   \psi \in C^1,
\end{equation*}
touches $u$ from above at a time-interior point $(x_0,t_0)$, then
$\psi'(t_0)\le f(x_0,t_0)$. For a supersolution, a lower contact by
$\psi(t)-A|x-x_0|^4$ gives $ \psi'(t_0)\ge f(x_0,t_0) $.
\end{lemma}

\begin{proof}
For $\gamma\ge0$, both the spatial gradient and the Hessian of the
quartic test vanish at $(x_0,t_0)$. The desired inequality therefore
follows directly from the standard viscosity definition and $F(0)=0$.
It remains to treat $-\frac{2}{3}<\gamma<0$. We prove the subsolution assertion
by contradiction. Translate the contact point to $(0,0)$ and normalize
the contact value so that
\[
 u(0,0)=\psi(0),  \quad
 u(x,t)\le\psi(t)+A|x|^4
\]
in a neighborhood of $(0,0)$. Suppose that
\[
 d_0:=\psi'(0)-f(0,0)>0.
\]

\smallskip

\noindent\emph{Step 1.}
Fix $\theta>0$ and set
\[
 \psi_\theta(t)=\psi(t)+\theta t^2,  \quad  B=A+\theta.
\]
Choose $r,s>0$ so that the inequality holds on
$K:=\overline B_r\times[-s,s]$, compactly contained in the
translated domain. Then
\[
 \Phi_0(x,t):=u(x,t)-\psi_\theta(t)-B|x|^4
 \le-\theta(t^2+|x|^4).
\]
Thus $\Phi_0$ has a unique maximum, of value zero, at $(0,0)$.
For each fixed $ y $ near $ 0 $, we let
\[
 \Phi_y(x,t):=u(x,t)-\psi_\theta(t)-B|x-y|^4,
\]
which, being upper semicontinuous, attains its maximum on $K$.

Now we claim that all its maximizing points tend to $(0,0)$ as $ y\to 0 $.

Indeed, take $y_j\to0$ and any corresponding maximizing points
$z_j=(x_j,t_j)$. Along a subsequence, $z_j\to z_*\in K$. By the upper semi-continuity, we have
\[
 0\le\liminf_{j \rightarrow \infty}\Phi_{y_j}(z_j)
 \le\limsup_{j \rightarrow \infty}\Phi_{y_j}(z_j)
 \le\Phi_0(z_*)\le0,
\]
where the first inequality follows from
$\Phi_{y_j}(z_j)\ge\Phi_{y_j}(0,0)=-B|y_j|^4$.
Hence $\Phi_0(z_*)=0$, that is to say, $ z_*=(0,0) $. This proves the claim.
In particular, for small $ |y| $, every maximizing point lies in
$B_{\frac{r}{2}}\times(-\frac{s}{2},\frac{s}{2})$.

\smallskip

\noindent\emph{Step 2.}
Let $(x_y,t_y)$ be any maximizing point of $\Phi_y$.
After adding a constant, the function
\[
 \phi_y(x,t)=\psi_\theta(t)+B|x-y|^4
\]
touches $u$ from above there. If $x_y\ne y$, put
$d:=|x_y-y|>0$. At this contact point, there holds
\[
 |D\phi_y|=4Bd^3,    \quad
 \|D^2\phi_y\|=12Bd^2.
\]
Condition (i) of Definition~\ref{Section2:def1}, together with
ellipticity and $F(0)=0$, therefore yields
\begin{equation}\label{Section2:eq1}
 \psi_\theta'(t_y)-f(x_y,t_y)
 \le |D\phi_y|^\gamma F(D^2\phi_y)
 \le C(n,\Lambda,\gamma,B)d^{2+3\gamma}.
\end{equation}
By {\it Step 1} and continuity, the left-hand side of \eqref{Section2:eq1} is at least $\frac{d_0}{2}$
for all sufficiently small $|y|$, uniformly over the maximizing
points. But $d\to0$ and $2+3\gamma>0$, so the right-hand side of \eqref{Section2:eq1} is less than $\frac{d_0}{2}$, which is impossible. Hence, there exists $\rho\in(0,\frac{r}{2})$ such that every maximizing point of $ \Phi_y $ has $ x_y=y $ whenever $ y\in B_\rho $.

\smallskip

\noindent\emph{Step 3.}
For $y\in B_\rho$, define
\[
 m(y):=\max_{t\in[-s,s]}\bigl(u(y,t)-\psi_\theta(t)\bigr).
\]
By {\it Step 2}, $\max_K\Phi_y(x,t)=m(y)$. Hence, for every
$x,y\in B_\rho$ and $t\in[-s,s]$,
\[
 u(x,t)-\psi_\theta(t)-B|x-y|^4\le m(y).
\]
Next maximizing in $t$ and then interchanging $x,y$ gives
\[
 |m(x)-m(y)|\le B|x-y|^4,
\]
which implies that $m$ is constant. In fact, divide the segment from $x$
to $y$ into $N$ equal pieces and sum the preceding inequality:
\[
 |m(x)-m(y)|\le N B\left(\frac{|x-y|}{N}\right)^4
 =\frac{B|x-y|^4}{N^3}\longrightarrow0.
\]
Since $m(0)=u(0,0)-\psi_\theta(0)=0$, we obtain
\[
 \sup_{t\in(-s,s)}\bigl(u(x,t)-\psi_\theta(t)\bigr)=0,
 \quad   x \in B_\rho.
\]
Here the open interval gives the same value: a maximizing time for
$m(x)$ is also part of a maximizing point of $\Phi_x$, and therefore
lies in $(-\frac{s}{2},\frac{s}{2})$ by {\it Step 1}. At $x=0$, the supremum is attained
at $t=0$. Condition (ii) now yields
\[
 \psi'(0)=\psi_\theta'(0)\le f(0,0),
\]
contradicting $d_0>0$. This proves the subsolution assertion.

For a supersolution, apply the result to $-u$, with $ -f $ and
operator $ F^\sharp(M):=-F(-M) $. The operator $ F^\sharp$ has the same ellipticity constants $ (\lambda, \Lambda) $, and a lower contact by $ \psi(t)-A|x-x_0|^4 $ becomes an upper contact for $ -u $.
\end{proof}

\subsection{Normalization and intrinsic scaling}

\subsubsection{Normalization}
Set
\begin{equation}\label{eq:FL-lambda-one}
 \delta:=\frac{\Lambda}{\lambda}-1,
 \qquad
 \overline F(\mathrm{M}):=\lambda^{-1}F(\mathrm{M}).
\end{equation}
Then $\overline F(0)=0$ and the ellipticity constants of
$\overline F$ are $1$ and $1+\delta$. The time change $s=\lambda t$,
with $U(x,s):=u(x,\frac{s}{\lambda})$, gives
\[
 U_s-|DU|^\gamma\overline F(D^2U)=\overline f(x,s)
 \quad\text{in }B_1\times(-\lambda,0],
 \qquad
 \overline f(x,s):=\lambda^{-1}f\!\left(x,\frac{s}{\lambda}\right).
\]
For the normalized estimates, we omit the bar from the operator and
write $F$ for an operator with ellipticity constants $1$ and $1+\delta$.
The Pucci inequalities then imply
\begin{equation*}
 \operatorname{tr}\mathrm{M}-\delta\operatorname{tr}\mathrm{M}^-
 \le F(\mathrm{M})
 \le \operatorname{tr}\mathrm{M}+\delta\operatorname{tr}\mathrm{M}^+,
 \qquad \mathrm{M}\in\mathrm{Sym}(n).
\end{equation*}
Here $\mathrm{M}^\pm$ are the positive and negative parts, so that
$\mathrm{M}=\mathrm{M}^+-\mathrm{M}^-$ and
$|\mathrm{M}|=\mathrm{M}^++\mathrm{M}^-$. Hence
\begin{equation}\label{eq:FL-close-trace}
 |F(\mathrm{M})-\operatorname{tr}\mathrm{M}|
 \le\delta\operatorname{tr}|\mathrm{M}|
 \le\mathrm{C}_n\delta\|\mathrm{M}\|.
\end{equation}
In Sections~\ref{Section 3}--\ref{Section 5}, we work with normalized
equations on $Q_1$ whose ellipticity constants are $1$ and $1+\delta$.
The reduction from the original equation to this setting is carried
out in Section~\ref{subsec:FL-normalization-comparison}; the time
change alone does not preserve $Q_1$.

\subsubsection{Intrinsic scaling} The following scaling argument will be used frequently in what follows. More precisely, for $ r,\mu>0 $, we define
\[
 Q_r^\mu:=B_r\times(-r^2\mu^{-\gamma},0].
\]
Suppose that $u$ solves \eqref{Intro:maineq1} in $Q_r^\mu$.
Let $L(x)=a+l\cdot x$ and define
\begin{equation*}
 v(x,t)
 :=\frac{u(rx,r^2\mu^{-\gamma}t)-L(rx)}{r\mu},
 \quad
 q:=\frac{l}{\mu}.
\end{equation*}
A direct computation gives
\[
 Du=l+\mu Dv=\mu(q+Dv),
 \quad
 D^2u=\frac{\mu}{r}D^2v,
 \quad
 u_t=\frac{\mu^{1+\gamma}}{r} v_t.
\]
Then $ v $ satisfies, in the viscosity sense,
\begin{equation*}
 v_t-|q+Dv|^\gamma F_{r,\mu}(D^2v)=g   \quad   \text{in}   \quad Q_{1}
\end{equation*}
with
\begin{equation*}
 g(x,t)
 :=r\mu^{-(1+\gamma)}
 f(rx,r^2\mu^{-\gamma}t),
\end{equation*}
and
\begin{equation*}
 F_{r,\mu}(\mathrm{M})
 :=\frac{r}{\mu} F\!\left(\frac{\mu}{r}\mathrm{M}\right).
\end{equation*}

\begin{lemma}\label{Section2.2:lemma1}
If $F$ satisfies \hyperref[F1]{\bf (F1)}, then $F_{r,\mu}$
satisfies the same condition with the same ellipticity constants. Moreover, under the normalization \eqref{eq:FL-lambda-one}, we have
\begin{equation*}
   |F_{r,\mu}(\mathrm{M})- \operatorname{tr} \mathrm{M}|
 \leq \mathrm{C}_n\delta\|\mathrm{M}\|.
\end{equation*}
\end{lemma}

\begin{proof}
For $ \mathrm{M}, \mathrm{N} \in \text{Sym}(n) $,
\begin{align*}
 F_{r,\mu}(\mathrm{M}+\mathrm{N})-F_{r,\mu}(\mathrm{M})
 &=\frac{r}{\mu}\left[
 F\!\left(\frac{\mu}{r}(\mathrm{M}+\mathrm{N})\right)
 -F\!\left(\frac{\mu}{r}\mathrm{M}\right)
 \right],
\end{align*}
Applying the Pucci inequalities to $\frac{\mu}{r}\mathrm{N} $ and using the positive one-homogeneity of $\mathscr P^\pm_{\lambda,\Lambda}$ gives
\[
 \mathscr P^-_{\lambda,\Lambda}(\mathrm{N})
 \le F_{r,\mu}(\mathrm{M}+\mathrm{N})-F_{r,\mu}(\mathrm{M})
 \le \mathscr P^+_{\lambda,\Lambda}(\mathrm{N}).
\]
The last assertion follows from \eqref{eq:FL-close-trace}, applied to
$F_{r,\mu}$, which also vanishes at zero.
\end{proof}

\section{Uniform compactness and heat approximation}
\label{Section 3}
The aim of this section is to approximate solutions of the shifted
equation by caloric functions when $|\gamma|$, the ellipticity
aperture, and the source term are sufficiently small. We first
establish uniform equicontinuity for bounded shifts using temporal
regularization, doubling of variables, and quartic barriers.
We then identify the heat equation in the limit, treating contacts
where the total gradient vanishes separately. This yields the
approximation lemma needed for the improvement of flatness
in Section~\ref{Section 4}.

\subsection{Temporal regularization}

\begin{lemma}[Temporal regularization]\label{lem:time-regularization}
Assume that $ F $ satisfies \hyperref[F1]{\bf (F1)}. Let $ w $ be a bounded continuous viscosity solution of
\[
 w_t-|q+Dw|^\gamma F(D^2w)= f
\]
in $Q_1$, where $q\in\mathbb R^n$, $\gamma>-1$, and
$f\in C^0(Q_1)\cap L^\infty(Q_1)$. On a compact spatial set
and a closed time interval $J$ whose product lies in
$B_1\times(-1,0)$, define
\begin{align}\label{eq:time-convolutions}
 w^\epsilon(x,t)&:=\sup_{\tau\in J}
 \left(w(x,\tau)-\frac{(t-\tau)^2}{2\epsilon}\right),\notag\\
 w_\epsilon(x,t)&:=\inf_{\tau\in J}
 \left(w(x,\tau)+\frac{(t-\tau)^2}{2\epsilon}\right).
\end{align}
On every smaller compact cylinder these functions converge uniformly to
$w$ and are Lipschitz in time for each fixed $\epsilon>0$. For sufficiently small $\epsilon$, at contacts with nonzero total gradient, they satisfy
\begin{equation}\label{eq:convolution-inequalities}
\begin{aligned}
 (w^\epsilon)_t-|q+Dw^\epsilon|^\gamma F(D^2w^\epsilon)
 &\le\|f\|_\infty,\\
 (w_\epsilon)_t-|q+Dw_\epsilon|^\gamma F(D^2w_\epsilon)
 &\ge-\|f\|_\infty.
\end{aligned}
\end{equation}
\end{lemma}

\begin{proof}
Write $ J:=[a,b] $ and $ N:=\|w\|_\infty $ on the cylinder on which
the convolutions are taken. A maximizing time $ \tau $ for $ w^\epsilon(x,t) $ satisfies
\[
 \frac{|t-\tau|^2}{2\epsilon}
 \le w(x,\tau)-w(x,t)\le 2N.
\]
The same estimate holds for a minimizing time in $w_\epsilon$.
Thus all optimizing times satisfy $|t-\tau|\le2\sqrt{N\epsilon}$.
On a smaller compact cylinder, uniform continuity of $w$ now gives
\[
 0\le w^\epsilon-w\le\omega(2\sqrt{N\epsilon}),
 \qquad
 0\le w-w_\epsilon\le\omega(2\sqrt{N\epsilon}),
\]
where $\omega$ is a modulus of continuity on the original compact
cylinder. This proves uniform convergence. For fixed $\epsilon$, the
functions $t\mapsto \frac{\mp(t-\tau)^2}{2\epsilon}$ have a common Lipschitz constant on $J$, uniformly over $\tau\in J$. Taking their supremum or infimum preserves this constant.

Suppose that $\phi$ touches $w^\epsilon$ from above at $(x_0,t_0)$,
and let $\tau_0$ be an optimizing time. For sufficiently small
$\epsilon$, the preceding estimate places $\tau_0$ in the interior
of $J$. Set $c:=t_0-\tau_0$. For $(x,s)$ near $(x_0,\tau_0)$,
\[
 w(x,s)-\frac{c^2}{2\epsilon}
 \le w^\epsilon(x,s+c)\le\phi(x,s+c),
\]
with equality at $(x_0,\tau_0)$ after the usual normalization of the
contact value. Consequently
\[
 (x,s)\longmapsto\phi(x,s+c)+\frac{c^2}{2\epsilon}
\]
is an upper test for $w$ there. Its time derivative, gradient, and
spatial Hessian equal those of $\phi$ at $(x_0,t_0)$. If
$q+D\phi(x_0,t_0)\ne0$, the subsolution inequality yields
\[
 \phi_t(x_0,t_0)-|q+D\phi(x_0,t_0)|^\gamma
 F(D^2\phi(x_0,t_0))\le f(x_0,\tau_0)\le\|f\|_\infty.
\]
For $ w_\epsilon $, use a minimizing time and a lower test; the same
translation gives the supersolution inequality in
\eqref{eq:convolution-inequalities}. These inequalities also hold
for closed parabolic semijets whose total gradient is nonzero:
apply the test inequalities to a defining sequence of ordinary
semijets and use continuity of the operator away from zero total
gradient.
\end{proof}

\subsection{A doubling-variable lemma}

\begin{lemma}\label{lem:doubling}
Let $w$ and $F$ satisfy the hypotheses of
Lemma~\ref{lem:time-regularization}, and put $H_0:=\|f\|_\infty$.
Choose a compact cylinder $K:=\overline B_R\times[a,T] \Subset Q_{1}$. Let $ L\ge 0 $, $ \zeta>0 $, and let
$\varphi\in C^{0}([0,2R])\cap C^2((0,2R)) $ satisfy
\[
 \varphi(0)=0, \quad   \varphi\ge0,  \quad
 \varphi'>0,  \quad   \varphi''<0    \quad   \text{on }(0,2R).
\]
For fixed $x_0,y_0\in B_R$ and $t_0\in(a,T)$, set
\[
 \psi(t)=\frac{L}{2}(t-t_0)^2+\frac{\zeta}{T-t}.
\]
Suppose that
\begin{equation}\label{eq:unregularized-functional}
 \Psi(x,y,t):=w(x,t)-w(y,t)-\varphi(|x-y|)
 -\frac{L}{2}\bigl(|x-x_0|^2+|y-y_0|^2\bigr)-\psi(t)
\end{equation}
has a positive maximum on $\overline B_R^2\times[a,T)$, and that all
maximizers lie in $B_R^2\times(a,T)$. At each maximizer, write
\[
 d:=|x-y|,\quad e:=\frac{x-y}{d},   \quad
 a_1=\varphi'(d)e+L(x-x_0),   \quad
 a_2=\varphi'(d)e-L(y-y_0).
\]
Assume $q+a_1\ne0$ and $q+a_2\ne0$ at every such maximizer.
Then there are a maximizing point $(\bar x,\bar y,\bar t)$ and
matrices $ \mathrm{X},\mathrm{Y} \in \text{Sym}(n) $ such that
\begin{equation}\label{eq:rigorous-subtraction}
 L(\bar t-t_0)+\frac{\zeta}{(T-\bar t)^2}
 \le |q+a_1|^\gamma F(\mathrm{X}+L \Id_n)-|q+a_2|^\gamma F(\mathrm{Y}-L\Id_n)+2H_0,
\end{equation}
and, with $d,e,a_1,a_2$ evaluated at this point,
\begin{equation}\label{eq:doubling-matrices}
 \mathrm{X}-\mathrm{Y} \le 0,  \quad
 \langle(\mathrm{X}-\mathrm{Y})e,e\rangle\le2\varphi''(d),  \quad
 \|\mathrm{X}\|+\|\mathrm{Y}\|\le \mathrm{C}\left(|\varphi''(d)|+\frac{\varphi'(d)}{d}\right).
\end{equation}
The constant $ \mathrm{C} $ is universal and independent of $ d $.
\end{lemma}

\begin{proof}
For this purpose, we divide the proof into the following three steps:

\emph{Step 1.} Let $w^\epsilon$ and $w_\epsilon$ be the time convolutions from Lemma~\ref{lem:time-regularization}, taken over an interval slightly larger than $[a,T]$. Replace $w(x,t)-w(y,t)$ in $\Psi$ by $ w^\epsilon(x,t)-w_\epsilon(y,t) $ and denote the resulting function by $ \Psi_\epsilon $.

The pole $\frac{\zeta}{T-t} $ keeps its maximizers a fixed distance from $T$. On the remaining compact set, $\Psi_\epsilon\to\Psi$ uniformly.
Since all maximizers of $\Psi$ are interior, for small $\epsilon$,
we can choose interior maximizers
$(x_\epsilon,y_\epsilon,t_\epsilon)$ such that, along a subsequence,
\[
 (x_\epsilon,y_\epsilon,t_\epsilon)
 \longrightarrow(\bar x,\bar y,\bar t),
 \quad
 \Psi(\bar x,\bar y,\bar t)=\max\Psi>0.
\]
We have $\bar x\ne\bar y$, since otherwise
$\Psi(\bar x,\bar y,\bar t)<0$, contradicting positivity. Moreover,
\begin{equation}\label{eq:distance-separated}
 d_\epsilon:=|x_\epsilon-y_\epsilon|
 \longrightarrow d:=|\bar x-\bar y| > 0.
\end{equation}

\medskip

\emph{Step 2.}
To keep the quadratic localizers out of the matrix calculation, set
\[
 U_\epsilon(x,t)
 :=w^\epsilon(x,t)-\frac{L}{2}|x-x_0|^2,
 \qquad
 V_\epsilon(y,t)
 :=w_\epsilon(y,t)+\frac{L}{2}|y-y_0|^2.
\]
Then
\[
 U_\epsilon(x,t)-V_\epsilon(y,t)
 -\varphi(|x-y|)-\psi(t)
\]
has a maximum at the chosen triple. For fixed $\epsilon$, both
functions are Lipschitz in time. Hence condition (8.5) of
\cite[Theorem~8.3]{CIL92} holds, and the parabolic theorem of sums
applies.

For the following calculation, write $ d:=d_\epsilon $ and
$e:=\frac{x_\epsilon-y_\epsilon}{d_\epsilon} $, and omit the matrix
subscripts. Put
\[
 S:=|\varphi''(d)|+\frac{\varphi'(d)}{d},
 \qquad \tau:=4S,
\]
and
\[
 \mathrm{Z}:=\varphi''(d)e\otimes e
   +\frac{\varphi'(d)}{d}(\Id_n-e\otimes e).
\]
The spatial Hessian of $\varphi(|x-y|)$ is
\[
  \mathrm{A}=
 \begin{pmatrix}
  \mathrm{Z} &-\mathrm{Z} \\
  -\mathrm{Z} & \mathrm{Z}
 \end{pmatrix},
 \quad
  \mathrm{A}^2=
 2\begin{pmatrix}
  \mathrm{Z}^2&- \mathrm{Z}^2\\
  -\mathrm{Z}^2  & \mathrm{Z}^2
 \end{pmatrix}.
\]
Choosing the matrix parameter in the theorem to be $\tau^{-1}$
gives
\[
 -(\tau+\| \mathrm{A}\|)\Id_{2n}
 \le
 \begin{pmatrix}
  \mathrm{X}  &  0   \\
  0  &-\mathrm{Y}
 \end{pmatrix}
 \le  \mathrm{A} + \tau^{-1} \mathrm{A}^2.
\]

For any $z\in\mathbb R^n$, we have $ \mathrm{A}(z,z)=0$. Testing the upper
inequality on $(z,z)$ therefore yields
\[
 \langle( \mathrm{X}- \mathrm{Y})z,z\rangle\le0,
\]
which proves $  \mathrm{X} -  \mathrm{Y}  \le 0 $.

To obtain the radial estimate, observe that
$ \mathrm{Z}e=\varphi''(d)e$. Testing the same inequality on $(e,-e)$ gives
\begin{equation}\label{eq:matrix-radial}
 \langle( \mathrm{X}- \mathrm{Y})e,e\rangle
 \le4\varphi''(d)
    +8\tau^{-1}\varphi''(d)^2
 \le2\varphi''(d).
\end{equation}
Indeed, writing $b:=-\varphi''(d)>0$, the expression on the right
of the first inequality is
\[
 -4b+\frac{2b^2}{S}\le-2b,
\]
since $S\ge b$.

Moreover, $\| \mathrm{A}\|\le2S$ and $\| \mathrm{A}^2\|\le4S^2$. The block inequality
thus places every eigenvalue of $ \mathrm{X} $ and $- \mathrm{Y} $ between $-6S $ and $ 3S $. In particular,
\[
 \| \mathrm{X}\|+\| \mathrm{Y}\| \le 12S.
\]

Restoring $\epsilon$ and adding back the quadratic localizers,
the closed parabolic semijets supplied by the theorem are
\begin{align*}
 (b_{1,\epsilon},a_{1,\epsilon}, \mathrm{X}_\epsilon+L\Id_n)
 &\in
 \overline{\mathcal P}^{2,+}
 w^\epsilon(x_\epsilon,t_\epsilon),\\
 (b_{2,\epsilon},a_{2,\epsilon}, \mathrm{Y}_\epsilon-L\Id_n)
 &\in
 \overline{\mathcal P}^{2,-}
 w_\epsilon(y_\epsilon,t_\epsilon),
\end{align*}
where $\overline{\mathcal P}^{2,+}$ and
$\overline{\mathcal P}^{2,-}$ denote the closed parabolic superjet
and subjet, respectively. Here
\[
 a_{1,\epsilon}
 =\varphi'(d_\epsilon)e_\epsilon+L(x_\epsilon-x_0),
 \qquad
 a_{2,\epsilon}
 =\varphi'(d_\epsilon)e_\epsilon-L(y_\epsilon-y_0),
\]
with $ e_\epsilon=\frac{x_\epsilon-y_\epsilon}{d_\epsilon} $, and
\[
 b_{1,\epsilon}-b_{2,\epsilon}=\psi'(t_\epsilon).
\]

\medskip
\emph{Step 3.}
The gradients $a_{i,\epsilon}$ converge to $a_i$ at the limiting
maximizer. Since $q+a_i\ne0$ by hypothesis, both total gradients
remain nonzero for small $\epsilon$. By Lemma~\ref{lem:time-regularization}, we have
\begin{align*}
 b_{1,\epsilon}
 &\le
 |q+a_{1,\epsilon}|^\gamma F( \mathrm{X}_\epsilon+L\Id_n)+H_0,\\
 b_{2,\epsilon}
 &\ge
 |q+a_{2,\epsilon}|^\gamma F( \mathrm{Y}_\epsilon-L\Id_n)-H_0.
\end{align*}
Subtracting these inequalities gives
\[
 \psi'(t_\epsilon)
 \le
 |q+a_{1,\epsilon}|^\gamma F( \mathrm{X}_\epsilon+L\Id_n)
 -
 |q+a_{2,\epsilon}|^\gamma F( \mathrm{Y}_\epsilon-L\Id_n)
 +2H_0.
\]

By \eqref{eq:distance-separated} and {\it Step~2}, the matrices are
uniformly bounded as $\epsilon\to0$. Take a further subsequence
such that
\[
  \mathrm{X}_\epsilon\to  \mathrm{X},
 \quad    \text{and}  \quad
  \mathrm{Y}_\epsilon\to  \mathrm{Y}.
\]
Continuity of $F$ and of $\xi\mapsto|\xi|^\gamma$ away from zero
then allows us to pass to the limit. From
\[
 \psi'(\bar t)
 =L(\bar t-t_0)+\frac{\zeta}{(T-\bar t)^2},
\]
we obtain \eqref{eq:rigorous-subtraction}. The matrix inequalities
also pass to the limit, giving \eqref{eq:doubling-matrices}.
\end{proof}

\subsection{Uniform equicontinuity}

\begin{lemma}\label{lem:FL-equicont}
Fix $M>0$ and $\gamma_0\in(0,\frac{1}{4}]$. There exists
$\mathrm{C}=\mathrm{C}(n,M,\gamma_0)$ such that the following holds.
Assume $|\gamma|\le\gamma_0$, $0\le\delta\le\frac{1}{2}$, and $|q|\le M$,
and let $F$ be $(1,1+\delta)$-elliptic with $F(0)=0$.
Every bounded continuous viscosity solution $w$ of
\[
w_t-|q+Dw|^\gamma F(D^2w)=f \quad\text{in }   Q_1,
\]
where $f\in C^0(Q_1)\cap L^\infty(Q_1)$ and
\begin{equation*}
  \norm w_{L^\infty(Q_1)}\le1,\quad\norm f_{L^\infty(Q_1)}\le 1,
\end{equation*}
satisfies
\begin{equation}\label{eq:FL-equicont}
|w(x,t)-w(y,s)|\le \mathrm{C} \bigl(|x-y|^{\frac{1}{2}}+|t-s|^{\frac{1}{10}}\bigr)
\end{equation}
for $ (x,t),(y,s)\in Q_{\frac{7}{8}} $.
\end{lemma}

\begin{proof}

\medskip
\noindent
\emph{Step 1: spatial estimate.}
Set $\theta:=\frac{1}{2}$ and $K:=65536$. We claim that, for a sufficiently large constant $ L $,
\[
 |w(x,t)-w(y,t)|\le L|x-y|^\theta
 \quad   \text{in }   Q_{\frac{15}{16}}.
\]
Suppose, to the contrary, that there
exist $x_0,y_0\in B_{\frac{15}{16}}$ and $t_0\in[-(\frac{15}{16})^2,0) $ such that
\[
 m_0:=w(x_0,t_0)-w(y_0,t_0)-L|x_0-y_0|^\theta > 0.
\]
Now we consider
\[
 \begin{aligned}
 \Psi(x,y,t):={}&w(x,t)-w(y,t)-L|x-y|^\theta\\
 &-\frac{K}{2}\bigl(
 |x-x_0|^2+|y-y_0|^2+|t-t_0|^2
 \bigr)-\frac{\zeta}{T-t}
 \end{aligned}
\]
on
$ \overline B_{\frac{31}{32}}^{\,2}\times[-(\frac{31}{32})^2,T) $, where $ t_0<T<0 $ and $   0<\zeta< \frac{m_0(T-t_0)}{2} $. The quadratic penalties exclude the lateral and bottom boundaries, while the time pole excludes $t=T$. Hence there exists an interior point $(\bar x,\bar y,\bar t)$ such that
\begin{equation*}
  \Psi(\bar x,\bar y,\bar t) =  \max_{(x,y,t) \in \overline B_{\frac{31}{32}}^{\,2}\times[-(\frac{31}{32})^2,T) } \Psi(x,y,t) > 0.
\end{equation*}

Write
\[
 d:=|\bar x-\bar y|,\qquad
 e:=\frac{\bar x-\bar y}{d},\qquad
 P:=L\theta d^{\theta-1}.
\]
Positivity gives $d>0$, $Ld^\theta\le2$, and
\[
 K|\bar x-x_0|+K|\bar y-y_0|\le \mathrm{C} \sqrt K.
\]
The two total test gradients are
\[
 \xi_1=q+Pe+K(\bar x-x_0),
 \qquad
 \xi_2=q+Pe-K(\bar y-y_0).
\]
Since $P\ge cL$, choosing $L$ sufficiently large gives
\[
 \frac{P}{2}\le|\xi_i|\le2P,   \quad
 |\xi_1-\xi_2|\le \mathrm{C} \sqrt K.
\]
These bounds hold at every maximizer, so
Lemma~\ref{lem:doubling} applies with
$\varphi(r)=Lr^\theta$ and localization coefficient $K$.
It provides matrices $\mathrm{X}, \mathrm{Y} $ satisfying
\[
 \mathrm{X}- \mathrm{Y} \le 0,\qquad
 \langle(\mathrm{X}-\mathrm{Y})e,e\rangle\le-cLd^{\theta-2},
 \qquad
 \|\mathrm{X}\|+\|\mathrm{Y}\|\le \mathrm{C}Ld^{\theta-2}.
\]

Put $A_i:=|\xi_i|^\gamma$. Uniform ellipticity implies
\[
 F(\mathrm{X})-F(\mathrm{Y})
 \le\mathscr P^+_{1,1+\delta}(\mathrm{X}-\mathrm{Y})
 =\operatorname{tr}(\mathrm{X}-\mathrm{Y})
 \le-\mathrm{c}Ld^{\theta-2}.
\]
The gradient bounds also give
\[
 A_i\asymp P^\gamma,\qquad
 |A_1-A_2|
 \le \mathrm{C}|\gamma|P^{\gamma-1}\sqrt K,
 \qquad
 |F(\mathrm{Y})|\le \mathrm{C}Ld^{\theta-2}.
\]
Using
\[
 A_1F(\mathrm{X})-A_2F(\mathrm{Y})
 =A_1\bigl(F(\mathrm{X})-F(\mathrm{Y})\bigr)+(A_1-A_2)F(\mathrm{Y})
\]
in \eqref{eq:rigorous-subtraction}, and estimating the quadratic
localization terms by ellipticity, we obtain
\begin{equation}\label{eq:FL-holder-contradiction}
 \begin{aligned}
 0\le{}&
 -\mathrm{c}P^\gamma Ld^{\theta-2}
 +\mathrm{C}|\gamma|P^{\gamma-1}\sqrt K\,Ld^{\theta-2}\\
 &+\mathrm{C} K(P^\gamma+1)+2.
 \end{aligned}
\end{equation}
We also used $K(\bar t-t_0)\ge-K$ and included this contribution in the $\mathrm{C} K$ term. The positive derivative of the time pole has been discarded.

The negative term dominates all the others when $L$ is large.
Indeed, the ratio of the second term to its magnitude is at most
$\frac{\mathrm{C}|\gamma|\sqrt{K}}{P}$, and the corresponding ratio for
$\mathrm{C}KP^\gamma$ is at most $\frac{\mathrm{C}}{L}$. Moreover, since $\theta=\frac{1}{2}$,
\begin{equation}\label{eq:FL-negative-domination}
 P^\gamma Ld^{\theta-2}
 =\theta^\gamma L^{1+\gamma}d^{-\frac{3+\gamma}{2}}
 \ge \mathrm{c}L^{1-\gamma_0}.
\end{equation}
Thus the remaining constant terms are also absorbed. This
contradicts \eqref{eq:FL-holder-contradiction} and proves the
spatial estimate for $t<0$. The estimate at $t=0$ follows by
continuity.

\medskip
\noindent
\emph{Step 2: Time estimate.}
Set $v(x,t):=w(x,t)+q\cdot x$. Then
\[
 v_t-|Dv|^\gamma F(D^2v)=f,
 \qquad \|v\|_\infty\le1+M,
\]
and the spatial estimate gives
\[
 |v(x,t)-v(y,t)|\le \mathrm{C}|x-y|^\theta
 \quad   \text{in }Q_{\frac{15}{16}}.
\]

Fix $x_0\in B_{\frac{7}{8}}$ and
$-(\frac{7}{8})^2\le s<t<0$. Let $R=\frac{1}{16}$. For
$0<\epsilon\le\epsilon_0<R$, define
\[
 A_\epsilon:= \mathrm{C}_A\epsilon^{\theta-4},
 \qquad
 a_0:=(4-\theta)(1+\gamma_0),
 \qquad
 B_\epsilon:= \mathrm{C}_B(1+\epsilon^{-a_0}),
\]
and consider the upper barrier
\[
 \Phi(x,\tau)
 :=v(x_0,s)+\mathrm{C}_A\epsilon^\theta
   +A_\epsilon|x-x_0|^4+B_\epsilon(\tau-s).
\]
Choose $ \mathrm{C}_A\ge1$ larger than the spatial H\"older constant.
For $r:=|x-x_0|\le\epsilon$, the term
$\mathrm{C}_A\epsilon^\theta$ controls the initial oscillation. For
$r\ge\epsilon$, we use
\[
 \epsilon^{\theta-4}r^4\ge r^\theta.
\]
Hence $v(x,s)\le\Phi(x,s)$ in $B_R(x_0)$.
By decreasing $\epsilon_0$, we also ensure that
$A_\epsilon R^4\ge2(1+M)$, so $\Phi\ge v$ on the lateral boundary.

For $0<r<R$,
\[
 |D\Phi|=4A_\epsilon r^3,
 \qquad
 \|D^2\Phi\|\le \mathrm{C}_nA_\epsilon r^2.
\]
Consequently,
\begin{equation}\label{eq:FL-time-barrier-diffusion}
 \bigl||D\Phi|^\gamma F(D^2\Phi)\bigr|
 \le \mathrm{C} A_\epsilon^{1+\gamma}r^{2+3\gamma}
 \le \mathrm{C}\epsilon^{-a_0}.
\end{equation}
We used $A_\epsilon\ge1$, $r<1$, and
$2+3\gamma\ge\frac{5}{4}$. Choosing $ \mathrm{C}_B$ sufficiently large gives
\[
 \Phi_t-|D\Phi|^\gamma F(D^2\Phi)>1
 \quad  \text{for }r>0.
\]
At $r=0$, Lemma~\ref{Section2:lemma1} rules out an upper
contact with $v$, since $\Phi_t=B_\epsilon>1\ge f$.
Thus a positive interior maximum of $v-\Phi$ is impossible.
This argument can be justified by adding a positive time pole
with pole time between $t$ and $0$, and then letting its
coefficient tend to zero. We conclude that
\[
 v(x_0,t)-v(x_0,s)
 \le \mathrm{C} \bigl(\epsilon^\theta
            +(t-s)\epsilon^{-a_0}\bigr).
\]
Applying the same argument to $-v$ and
$F^\sharp(\mathrm{N}):=-F(-\mathrm{N})$ gives the reverse inequality.
Since $v-w=q\cdot x$ is independent of time,
\[
 |w(x_0,t)-w(x_0,s)|
 \le \mathrm{C} \bigl(\epsilon^\theta
            +(t-s)\epsilon^{-a_0}\bigr).
\]

For sufficiently small $t-s$, choose
$\epsilon=(t-s)^{\frac{1}{\theta+a_0}}$. It follows that
\[
 |w(x_0,t)-w(x_0,s)|\le \mathrm{C}|t-s|^\sigma,
 \quad
 \sigma:=\frac{\theta}{\theta+(4-\theta)(1+\gamma_0)}
 \ge\frac{4}{39}>\frac{1}{10}.
\]
Boundedness handles larger time differences, and continuity
includes $t=0$. Finally,
\[
 \begin{aligned}
 |w(x,t)-w(y,s)|
 &\le |w(x,t)-w(y,t)|+|w(y,t)-w(y,s)|\\
 &\le \mathrm{C} \bigl(|x-y|^{\frac{1}{2}}+|t-s|^{\frac{1}{10}}\bigr),
 \end{aligned}
\]
as required.
\end{proof}

\subsection{Heat approximation}

\begin{lemma}\label{lem:recognition}
Let $v$ be continuous in an open parabolic cylinder. Assume the following
upper-contact properties at time-interior points:
\begin{enumerate}
\item If a $C^{2,1}$ function $\phi$ touches $v$ from above and $D\phi\ne0$,
then $\phi_t-\Delta\phi\le0$.
\item If $\psi(t)+A|x-x_0|^4$ touches $v$ from above at $(x_0,t_0)$,
where $A>0$ and $\psi$ is smooth, then $\psi'(t_0)\le0$.
\end{enumerate}
If the same properties hold for $-v$, then $v$ is caloric.
\end{lemma}

\begin{proof}

We prove the subsolution inequality. It is enough first to consider
smooth test functions. Indeed, a $C^{2,1}$ test can be made strict by
adding $|x-x_0|^4+(t-t_0)^2$, and then mollified locally. The local
maximizers of $v$ minus the mollified tests approach the original
contact point. The time derivatives and the first two spatial
derivatives of the tests converge there. Once the heat inequality
has been proved for smooth tests, passage to the limit therefore
gives it for all $C^{2,1}$ tests.

Let a smooth $\phi$ touch $v$ from above at an interior point,
translated to $(0,0)$, and suppose $D\phi(0,0)=0$. Normalize the
contact value to zero, and write
\[
 b=\phi_t(0,0), \quad \mathrm{Q}:=D^2\phi(0,0).
\]
Assume, for a contradiction, that $b-\tr \mathrm{Q} >0$. Choose
$\varepsilon>0$ so small that
\[
 \mathrm{Q}_\varepsilon:= \mathrm{Q} +\varepsilon \Id_n \quad\hbox{is invertible},
 \qquad b-\tr \mathrm{Q}_\varepsilon>0.
\]
Only finitely many positive values of $\varepsilon$ make
$\mathrm{Q}+\varepsilon \Id_n $ singular, so this choice is possible. Taylor's
formula and Young's inequality imply, after shrinking the
neighborhood, that
\[
 \phi(x,t)\le bt+ \mathrm{C}_\varepsilon t^2
                  +\frac{1}{2}\langle \mathrm{Q}_\varepsilon x,x\rangle.
\]
For completeness, the mixed Taylor terms are bounded by
$\frac{\varepsilon|x|^2}{4}+ \mathrm{C}_\varepsilon t^2$, while the purely spatial
remainder is bounded by $\frac{\varepsilon|x|^2}{4}$ on a sufficiently
small neighborhood. Fix $\theta>0$ and put
\[
 \psi(t)=bt+(\mathrm{C}_\varepsilon+\theta)t^2,
 \quad
 P(x)=\frac{1}{2}\langle \mathrm{Q}_\varepsilon x,x\rangle+\theta|x|^4.
\]
On a closed cylinder $ K:=\overline B_r\times[-s,s] $ contained in
the original domain,
\[
 v(x,t)-\psi(t)-P(x)\le-\theta(t^2+|x|^4),
\]
and equality holds at $(0,0)$. The maximum is therefore strict.
Choose $r,s$ still smaller, if necessary, so that
\[
 \psi'(t)-\Delta P(z)>0
 \quad\hbox{for } |t|\le s,\ |z|\le2r,
 \quad
 DP(z)=0,\ |z|\le2r\ \Longrightarrow\ z=0.
\]
The first assertion follows from $ b-\tr \mathrm{Q}_\varepsilon > 0$.
For the second, note that
$ DP(z)=(\mathrm{Q}_\varepsilon+4\theta|z|^2\Id_n)z $ and use the
invertibility of $ \mathrm{Q}_\varepsilon $.

For $y$ near zero, maximize
\[
 v(x,t)-\psi(t)-P(x-y)
\]
on $K$. Uniform convergence of $P(\cdot-y)$ to $P$, together
with the strict maximum just obtained, shows that every maximizing
point $(x_y,t_y)$ tends to $(0,0)$ as $y\to0$. In particular all
such maxima are interior when $|y|$ is sufficiently small. At a
maximizer, a constant translate of $\psi(t)+P(x-y)$ is an upper
test. If $DP(x_y-y)\ne0$, property (1) gives
\[
 \psi'(t_y)-\Delta P(x_y-y)\le0,
\]
contrary to the preceding strict inequality. Hence
$DP(x_y-y)=0$ and therefore $x_y=y$.

Choose $\rho>0$ so that these conclusions hold for every
$y\in B_\rho$. Set
\[
 m(y)=\max_{|t|\le s}\{v(y,t)-\psi(t)\}.
\]
Since every maximum of the shifted functional occurs at $x=y$,
its maximum value is precisely $m(y)$. Thus, for
$x,y\in B_\rho$,
\[
 m(x)-m(y)\le P(x-y).
\]
Interchanging $x,y$ and using that $P$ is even gives
\[
 |m(x)-m(y)|\le P(x-y)\le \mathrm{C} |x-y|^2.
\]
In particular $P(x-y)\ge0$. Subdivision of the segment from $x$
to $y$ into $N$ equal parts gives
$|m(x)-m(y)|\le \frac{\mathrm{C}|x-y|^2}{N} $, hence $ m $ is constant.
The original strict contact gives $m(0)=0$, so
\[
 v(x,t)\le\psi(t), \quad   v(0,0)=\psi(0)=0,
 \quad x\in B_\rho, \ |t|\le s.
\]
Property (2), applied to $\psi(t)+A|x|^4$ for any $A>0$,
now implies $b=\psi'(0)\le0$. On the other hand,
$P(z)\ge0$ for all small $z$ implies $ \mathrm{Q}_\varepsilon\ge0$,
by dividing $P(he)\ge0$ by $h^2$ and letting $h \to 0$.
It follows that $b-\tr \mathrm{Q}_\varepsilon \le 0 $, the desired
contradiction. Therefore $\phi_t-\Delta\phi\le0$ also at
zero-gradient contacts.

Applying this argument to $-v$ gives the supersolution inequality.
Thus $v$ is a viscosity solution of the heat equation. The standard
equivalence between viscosity and classical solutions for the heat
equation shows that $v$ is caloric.
\end{proof}

\begin{lemma}[Heat approximation]\label{lem:FL-heat}
Fix $ M>0 $ and $ \eta>0 $. There exists $\varepsilon_h=\varepsilon_h(n,M,\eta)>0$ such that the following holds. Let $|q|\le M$, let $F$ be $(1,1+\delta)$-elliptic with $F(0)=0$, and let $w$ be
a continuous viscosity solution of
\[
w_t-|q+Dw|^\gamma F(D^2w)=f  \quad  \text{in } Q_1.
\]
Assume
\[
\norm w_{L^\infty(Q_1)}\le1,
\quad  \gamma > -1,  \quad
|\gamma|+\delta+\norm f_{L^\infty(Q_1)}\le\varepsilon_h.
\]
Then there is a caloric function $H$ in $Q_{\frac{2}{3}}$ such that
\begin{equation}\label{eq:FL-heat-close}
\norm{w-H}_{L^\infty(Q_{\frac{1}{2}})}\le\eta,
\quad \norm H_{C^{2,1}(Q_{\frac{1}{2}})}\le \mathrm{C}(n).
\end{equation}
\end{lemma}

\begin{proof}
Suppose that the conclusion fails for some fixed $M$ and $\eta$.
There then exist solutions $w_j$, shifts $|q_j|\le M$, operators
$F_j$, exponents $\gamma_j$, and sources $f_j$, satisfying
\[
 \|w_j\|_\infty\le1,
 \quad
 |\gamma_j|+\delta_j+\|f_j\|_\infty\longrightarrow0,
\]
for which no $H$ has the stated properties. After discarding
finitely many indices we have $|\gamma_j|\le\frac{1}{4}$ and
$\delta_j\le\frac{1}{2}$. Lemma~\ref{lem:FL-equicont} and the
Arzel\`a--Ascoli theorem give, after passage to a subsequence,
\[
 q_j\longrightarrow q_\infty,
 \quad
 w_j\longrightarrow w_\infty
 \quad\hbox{uniformly on }\overline Q_{\frac{3}{4}}.
\]
Set $v_j(x,t)=w_j(x,t)+q_j\cdot x$ and
$v_\infty(x,t)=w_\infty(x,t)+q_\infty\cdot x$. Then
\[
 (v_j)_t-|Dv_j|^{\gamma_j}F_j(D^2v_j)=f_j.
\]
We verify the two upper-contact properties of
Lemma~\ref{lem:recognition} for $v_\infty$ in the
time-interior of $Q_{\frac{3}{4}}$.

First let $\phi$ be an upper test with $D\phi(x_0,t_0)\ne0$.
Make the contact strict by adding a positive multiple of
$|x-x_0|^4+(t-t_0)^2$. Uniform convergence produces upper
contacts for $v_j$ at points $(x_j,t_j)\to(x_0,t_0)$, with
test functions differing from $\phi$ only by constants. Their
gradients stay in a compact annulus. Hence
\[
 |D\phi(x_j,t_j)|^{\gamma_j}\longrightarrow1,
 \qquad
 F_j(D^2\phi(x_j,t_j))
       \longrightarrow\Delta\phi(x_0,t_0),
\]
where the second convergence uses
\eqref{eq:FL-close-trace}. Passing to the limit in the
subsolution inequality gives $\phi_t-\Delta\phi\le0$.

Next suppose that $\psi(t)+A|x-x_0|^4$ is an upper test
for $v_\infty$ at $(x_0,t_0)$, with $A>0$ and $\psi$
smooth. Fix $\theta>0$, put $B:=A+\theta$, and use the strict
upper test
\[
 \Phi(x,t)=\psi(t)+\theta(t-t_0)^2+B|x-x_0|^4.
\]
As before, it induces upper contacts for $v_j$ at
$(x_j,t_j)\to(x_0,t_0)$. Put $r_j=|x_j-x_0|$.
If $r_j>0$, the nonzero-gradient subsolution inequality and
uniform ellipticity give
\[
 \psi'(t_j)+2\theta(t_j-t_0)
 \le \|f_j\|_\infty
       +\mathrm{C}_{n,B}r_j^{2+3\gamma_j}
 \le \|f_j\|_\infty+ \mathrm{C}_{n,B}r_j^{\frac{5}{4}},
\]
for all sufficiently large $j$. Indeed,
$|D\Phi|=4Br_j^3$, $\|D^2\Phi\|\le \mathrm{C}_nBr_j^2$,
and $2+3\gamma_j\ge\frac{5}{4}$. If $r_j=0$,
Lemma~\ref{Section2:lemma1}, applied at $(x_0,t_j)$,
instead gives
\[
 \psi'(t_j)+2\theta(t_j-t_0)\le f_j(x_0,t_j).
\]
Both cases yield $\psi'(t_0)\le0$ after passage to the limit.
This argument is uniform even when the contact gradients tend to
zero arbitrarily rapidly.

Replacing $v_j$ by $-v_j$, $f_j$ by $-f_j$, and $F_j$ by
$F_j^\sharp(\mathrm{N})=-F_j(-\mathrm{N})$ verifies the same two properties for
$-v_\infty$. Lemma~\ref{lem:recognition} shows that
$v_\infty$, and therefore $w_\infty$, is caloric. Set
$H:=w_\infty$ on $Q_{\frac{2}{3}}$. Since $\|H\|_\infty\le1$,
the interior estimates for the heat equation give
\[
 \|H\|_{C^{2,1}(Q_{\frac{1}{2}})}\le \mathrm{C}(n).
\]
Uniform convergence gives
$\|w_j-H\|_{L^\infty(Q_{\frac{1}{2}})}\le\eta$ for all sufficiently
large $j$, a contradiction.
\end{proof}

\section{One-step intrinsic improvement of flatness}\label{Section 4}

We use heat approximation to improve the affine error on a smaller
intrinsic cylinder. The choice of exponents below also controls the
source under iteration and the final time regularity.

Fix the target exponent $ \alpha\in(0,1) $ and put
\begin{equation}\label{eq:FL-beta}
 \beta:=\frac{1+\alpha}{2}.
\end{equation}
Choose $\gamma_{\rm FL}>0$ so small that
\begin{equation}\label{eq:FL-beta-gamma}
 0<\gamma_{\rm FL}\le\frac{1}{4},
 \qquad \beta(1+\gamma_{\rm FL})<1,
\end{equation}
and
\begin{equation}\label{eq:FL-time-conversion}
 \frac{\beta}{2+\beta\gamma_{\rm FL}}\ge\frac{\alpha}{2},
 \qquad
 \frac{1+\beta}{2+\beta\gamma_{\rm FL}}\ge\frac{1+\alpha}{2}.
\end{equation}
For example, choose a positive $\gamma_{\rm FL}$ strictly smaller than
\[
 \min\left\{\frac{1}{4},\frac{1-\beta}{\beta},
 \frac{2(\beta-\alpha)}{\alpha\beta},
 \frac{2(\beta-\alpha)}{(1+\alpha)\beta}\right\}.
\]
All four bounds are positive because $\alpha<\beta<1$. We work throughout
with $-\gamma_{\rm FL}\le\gamma\le\gamma_{\rm FL}$.

\begin{lemma}[One-step intrinsic flatness]\label{lem:FL-onestep}
Fix $M>0$ and $\beta= \frac{1+\alpha}{2} $. Choose $\gamma_{\rm FL}\in(0,\frac{1}{4}]$ satisfying \eqref{eq:FL-beta-gamma}. There exist $\rho=\rho(n,\beta,\gamma_{\rm FL})\in(0,\frac{1}{8})$ and $ \mathrm{C}_0=\mathrm{C}_0(n)$ such that for every fixed $\eta_0\in(0,1]$ there is $\varepsilon_1=\varepsilon_1(n,M,\beta,\gamma_{\rm FL},\eta_0)>0$ with the following property. Let $F$ be $(1,1+\delta)$-elliptic with $F(0)=0$, and let $w$
be a continuous viscosity solution of
\[
w_t-|q+Dw|^\gamma F(D^2w)=f \quad\text{in }Q_1,
\quad |q|\le M,\quad\norm w_{L^\infty(Q_1)}\le\eta_0,
\]
with $\gamma > -1 $ and $|\gamma|+\delta+\norm f_\infty\le\varepsilon_1$. Then an affine function $\ell(x):=a+b\cdot x$ satisfies
\begin{equation}\label{eq:FL-one-flat}
\sup_{Q_\rho^{\rho^\beta}}|w-\ell|\le\eta_0\rho^{1+\beta},
\qquad |a|+|b|\le \mathrm{C}_0\eta_0.
\end{equation}
The scale $\rho$ does not depend on $M$ or $\eta_0$.
\end{lemma}

\begin{proof}
Put $\kappa=1-\beta(1+\gamma_{\rm FL})>0$. First choose $\rho$ so that $Q_\rho^{\rho^\beta}\subset Q_{\frac{1}{4}}$ for every $|\gamma|\le\gamma_{\rm FL}$, and
\[
\mathrm{C}_H(\rho^2+\rho^{2-\beta\gamma_{\rm FL}})\le\frac{1}{2}\rho^{1+\beta}.
\]
This is possible since $2-\beta\gamma_{\rm FL}=1+\beta+\kappa$. Next fix $\eta_0$, set $\eta_{\rm app}:=\frac{\eta_0\rho^{1+\beta}}{8}$, and choose $\varepsilon_1\le\gamma_{\rm FL}$ below the threshold in Lemma~\ref{lem:FL-heat} for this tolerance. That lemma supplies caloric $H$ with $\norm{w-H}_{L^\infty(Q_{\frac{1}{2}})}\le\eta_{\rm app}$ and $\norm H_{L^\infty(Q_{\frac{1}{2}})}\le2\eta_0$. Heat interior estimates give
\[
|H(0,0)|+|DH(0,0)|+\norm{D^2H}_{L^\infty(Q_{\frac{1}{4}})}
+\norm{H_t}_{L^\infty(Q_{\frac{1}{4}})}\le \mathrm{C}_H\eta_0.
\]
Taking $\ell(x):=H(0,0)+DH(0,0)\cdot x$ and using Taylor's theorem yields
\[
\sup_{Q_\rho^{\rho^\beta}}|H-\ell|
\le \mathrm{C}_H\eta_0(\rho^2+\rho^{2-\beta\gamma})
\le\frac{1}{2}\eta_0\rho^{1+\beta}.
\]
Adding the approximation error proves \eqref{eq:FL-one-flat}; the
derivative estimate bounds $ a,b $. In particular, $\rho$ is fixed
before $\eta_0$ is chosen.
\end{proof}

\section{The critical iteration and the stopping regime}\label{Section 5}

Set
\[
 r_k:=\rho^k,
 \qquad
 \mu_k:=r_k^\beta.
\]
We work first with a normalized continuous viscosity solution of
\eqref{Intro:maineq1} in $Q_1$ satisfying
\begin{equation}\label{eq:FL-normalized}
 u(0,0)=0,
 \qquad
 \osc_{Q_1}u\le\eta,
 \qquad
 |\gamma|+\delta+\|f\|_{L^\infty(Q_1)}\le\varepsilon_0,
\end{equation}
where $\eta$ and $\varepsilon_0$ are chosen in the proof of
Proposition~\ref{prop:FL-pointwise}. Set $K_*:=3$ and
\[
K^*:=\rho^{-\beta}(K_*+ \mathrm{C}_0).
\]
Both $K_*$ and $ K^* $ are independent of $ \eta $.

\begin{proposition}[Critical iteration]\label{prop:FL-iteration}
Start with $L_0\equiv0$.  The affine functions
\[
 L_k(x)=a_k+l_k\cdot x
\]
are constructed step by step as follows.  Whenever the level-$k$ flatness
estimate \eqref{eq:FL-flat-k} holds and the level-$k$ critical condition
\begin{equation}\label{eq:FL-critical}
 |l_k|\le K_*\mu_k
\end{equation}
is satisfied, one can construct $L_{k+1}$ so that
\begin{equation}\label{eq:FL-flat-k}
 \|u-L_k\|_{L^\infty(Q_{r_k}^{\mu_k})}
 \le \eta r_k\mu_k
 =\eta r_k^{1+\beta}
\end{equation}
also holds with $k+1$ in place of $k$, and
\begin{equation}\label{eq:FL-increments}
 |a_{k+1}-a_k|
 \le \mathrm{C}_0\eta r_k\mu_k,
 \quad
 |l_{k+1}-l_k|
 \le \mathrm{C}_0\eta\mu_k.
\end{equation}
The level-zero flatness follows from \eqref{eq:FL-normalized}.  In
particular, if $k$ is the first level where \eqref{eq:FL-critical} fails,
the function $L_k$ has already been constructed and its flatness estimate
remains valid.
\end{proposition}

\begin{proof}
Assume the construction has been made at level $k$ and define
\[
 w_k(x,t)
 :=\frac{
 u(r_kx,r_k^2\mu_k^{-\gamma}t)-L_k(r_kx)
 }{r_k\mu_k}.
\]
Then $\|w_k\|_\infty\le\eta$.  By
Lemma~\ref{Section2.2:lemma1}, $w_k$ satisfies
\[
 (w_k)_t-|q_k+Dw_k|^\gamma F_k(D^2w_k)=f_k,
 \quad
 q_k:=\frac{l_k}{\mu_k},
\]
where $F_k$ has the same ellipticity aperture as $F$, and
\begin{equation}\label{eq:FL-fk}
 f_k(x,t)
 :=r_k\mu_k^{-(1+\gamma)}
 f(r_kx,r_k^2\mu_k^{-\gamma}t).
\end{equation}
Since
\[
 r_k\mu_k^{-(1+\gamma)}
 =r_k^{1-\beta(1+\gamma)}\le1
\]
by \eqref{eq:FL-beta-gamma}, we have
$\|f_k\|_\infty\le\varepsilon_0$.  If
$|q_k|\le K_*$, Lemma~\ref{lem:FL-onestep}, applied with the fixed normalization level
$\eta_0=\eta$, yields an affine map
\[
 \widetilde\ell_k(x)=A_k+B_k\cdot x,
 \qquad
 |A_k|+|B_k|\le \mathrm{C}_0\eta,
\]
with
\[
 \|w_k-\widetilde\ell_k\|_{L^\infty(Q_\rho^{\rho^\beta})}
 \le\eta\rho^{1+\beta}.
\]
Set
\[
 a_{k+1}:=a_k+r_k\mu_kA_k,
 \qquad
 l_{k+1}:=l_k+\mu_kB_k.
\]
Since
\[
 r_k^2\mu_k^{-\gamma}\rho^2(\rho^\beta)^{-\gamma}
 =r_{k+1}^2\mu_{k+1}^{-\gamma},
\]
scaling back gives \eqref{eq:FL-flat-k} at level $k+1$, and
\eqref{eq:FL-increments} follows.
\end{proof}

\subsection{Infinite critical iteration}
There are two alternatives. If \eqref{eq:FL-critical} holds for every
$k$, then $|l_k|\le K_*\mu_k\to0$.  Summing the increments gives
$a_k\to u(0,0)=0$ and $l_k\to0$, while
\eqref{eq:FL-flat-k} yields
\begin{equation}\label{eq:FL-critical-expansion}
 |u(x,t)-u(0,0)|
 \le \mathrm{C} r^{1+\beta}
\end{equation}
whenever $|x|\le r$ and $|t|\le r^{2-\beta\gamma}$.  For the uniform
pointwise estimates in the two alternatives, set
\begin{equation}\label{eq:FL-comparison-cylinder}
 \vartheta_\gamma:=\begin{cases}
  2,&\gamma\ge0,\\
  2-\beta\gamma,&\gamma<0,
 \end{cases}
 \qquad
 \mathcal C_r(z_0):=B_r(x_0)\times(t_0-r^{\vartheta_\gamma},t_0].
\end{equation}
For $r<1$, $\mathcal C_r\subset Q_r^{r^\beta}$ for both signs of
$\gamma$, and consequently
\begin{equation}\label{eq:FL-critical-standard}
 \sup_{\mathcal C_r}|u-u(0,0)|
 \le \mathrm{C}r^{1+\beta}.
\end{equation}
In particular $Du(0,0)=0$.

\subsection{The first stopping scale}
Suppose instead that the iteration stops for the first time at level $k$.
Then
\[
 |q_k|:=\frac{|l_k|}{\mu_k}>K_*,
 \qquad
 |q_{k-1}|\le K_*.
\]
By \eqref{eq:FL-increments},
\begin{equation}\label{eq:FL-q-upper}
 |q_k|
 \le \rho^{-\beta}(K_*+ \mathrm{C}_0\eta) \le K^*.
\end{equation}
Thus the first stopping shift belongs to a fixed annulus.

\subsection{A uniform noncritical Lipschitz estimate}\label{sec:noncritical}

To use uniformly parabolic estimates at the stopping scale, we need
$|q+Dw|$ to stay away from zero. The following lemma gives this at
every test contact by proving that $w$ has spatial Lipschitz constant
at most one.

\begin{lemma}[Noncritical Lipschitz estimate]\label{lem:FL-noncritical}
Fix $K^*>2$.  There exist
$\varepsilon_*>0$ and $\eta_*>0$, depending only on $n$ and $K^*$, such
that the following holds.  Assume
\[
 |\gamma|\le\varepsilon_*,
 \qquad
 0\le\delta\le\varepsilon_*,
 \qquad
 2\le|q|\le K^*,
\]
and let $F$ be $(1,1+\delta)$-elliptic with $F(0)=0$. If $w$
is a bounded continuous viscosity solution of
\[
 w_t-|q+Dw|^\gamma F(D^2w)=f
 \quad\hbox{in }Q_1
\]
and
\[
 \|w\|_{L^\infty(Q_1)}+\|f\|_{L^\infty(Q_1)}\le\eta_*,
\]
then, for every fixed $t\in[-(\frac{3}{4})^2,0]$,
\begin{equation}\label{eq:FL-unit-lip}
 |w(x,t)-w(y,t)|\le|x-y|,
 \qquad x,y\in B_{\frac{3}{4}}.
\end{equation}
Consequently every spatial gradient $p$ of a viscosity test function touching
$w$ from above or below in $Q_{\frac{3}{4}}$ satisfies $|p|\le1$, and therefore
\begin{equation}\label{eq:FL-annulus}
 1\le|q+p|\le K^*+1.
\end{equation}
\end{lemma}

\begin{proof}
\smallskip
\noindent\emph{Step 1: a small H\"older modulus.}
Impose first $\varepsilon_*\le\frac{1}{4}$ and $\eta_*\le1$.
Lemma~\ref{lem:FL-equicont}, with $M=K^*$ and $\gamma_0=\frac{1}{4}$,
gives, for a fixed $\theta_0\in(0,1)$,
\[
 |w(x,t)-w(y,t)|\le \mathrm{C}_H|x-y|^{\theta_0}
\]
on $B_{\frac{7}{8}}$ at each fixed admissible time.  Together with
$\osc w\le2\eta_*$ this implies, after lowering the exponent to
\[
 \bar\theta:=\frac{\theta_0}{2},
\]
that
\begin{equation}\label{eq:FL-small-holder}
 |w(x,t)-w(y,t)|
 \le H|x-y|^{\bar\theta},
 \qquad
 H:=(2\eta_*)^{\frac{1}{2}}C_H^{\frac{1}{2}}.
\end{equation}
Hence $H\to0$ as $\eta_* \to 0$.

\smallskip
\noindent\emph{Step 2: localization.}
Let
\[
 \nu:=1+\frac{\bar\theta}{4}\in(1,2),
 \qquad
 \varphi(s):=s-\kappa_0s^\nu,
\]
then choose $\kappa_0>0$ so small that
$\varphi\in C^1([0,2])\cap C^2((0,2])$ and
\begin{equation}\label{eq:FL-phi}
 \frac{3}{4}\le\varphi'(s)\le1\quad(0\le s\le2),
 \quad
 \varphi''(s)=-c_\nu s^{\nu-2}\quad(0<s\le2).
\end{equation}
Suppose that \eqref{eq:FL-unit-lip} fails. By continuity, and after
interchanging the two points if necessary, there are
$x_0,y_0\in B_{\frac{3}{4}}$ and $t_0\in[-(\frac{3}{4})^2,0)$ for which
\[
 m_0:=w(x_0,t_0)-w(y_0,t_0)-|x_0-y_0|>0.
\]
Choose $t_0<T<0$ and $0<\zeta<\frac{m_0(T-t_0)}{2}$. Let
$L_1=4^5\eta_*$ and maximize
\begin{align*}
 \Psi(x,y,t):={}&w(x,t)-w(y,t)-\varphi(|x-y|)\\
 &-\frac{L_1}{2}\bigl(|x-x_0|^2+|y-y_0|^2+|t-t_0|^2\bigr)
 -\frac{\zeta}{T-t}
\end{align*}
over $\overline B_{\frac{7}{8}}^2\times[-(\frac{7}{8})^2,T)$.
Since $\varphi(s)\le s$, its value at $(x_0,y_0,t_0)$ exceeds
$\frac{m_0}{2}$. Its positive maximum lies away from the lateral and bottom
boundaries by the choice of $L_1$, and away from $t=T$ by the time pole.
For fixed $\zeta$, use the regularization in Lemma~\ref{lem:doubling}. Below $(\bar x,\bar y,\bar t)$ denotes
a limiting positive maximizer obtained as $\epsilon \to 0$.
The gradient bounds established below verify the nonzero-gradient
hypothesis at every positive maximizer and justify
\eqref{eq:rigorous-subtraction}. Set
\[
 s:=|\bar x-\bar y|,
 \qquad
 e:=\frac{\bar x-\bar y}{s}.
\]
Then $s>0$.  From \eqref{eq:FL-small-holder} and positivity of $\Psi$,
\[
 \frac{L_1}{2}\bigl(
 |\bar x-x_0|^2+|\bar y-y_0|^2
 \bigr)
 \le Hs^{\bar\theta},
\]
and therefore
\begin{equation}\label{eq:FL-localizer-grad}
 L_1|\bar x-x_0|+L_1|\bar y-y_0|
 \le \mathrm{C}\sqrt{L_1H}\,s^{\frac{\bar\theta}{2}}
 \le \mathrm{C} Hs^{\frac{\bar\theta}{2}},
\end{equation}
because $L_1=4^5\eta_*$, $H=(2C_H\eta_*)^{\frac{1}{2}}$, and
$\eta_*\le1$. The passage from the regularized maxima is justified
by \eqref{eq:distance-separated}.

\smallskip
\noindent\emph{Step 3: test gradients and matrix bounds.}
The spatial test gradients are
\[
 a_1=\varphi'(s)e+L_1(\bar x-x_0),
 \qquad
 a_2=\varphi'(s)e-L_1(\bar y-y_0).
\]
By \eqref{eq:FL-phi} and \eqref{eq:FL-localizer-grad},
$|a_i|\le \frac{9}{8}$ when $\eta_*$ is small. Since $|q|\ge2$, the total
gradients $\xi_i=q+a_i$ satisfy
\begin{equation}\label{eq:FL-fixed-annulus}
 \frac{3}{4}\le|\xi_i|\le K^*+\frac{3}{2},
\end{equation}
and
\begin{equation}\label{eq:FL-xi-diff}
 |\xi_1-\xi_2|\le CHs^{\frac{\bar\theta}{2}}.
\end{equation}
These bounds hold at every positive maximizer. We may therefore apply
Lemma~\ref{lem:doubling}, with localization coefficient $L_1$.
Its matrices satisfy
\begin{equation}\label{eq:FL-XY}
 \mathrm{X}- \mathrm{Y}\le0,\qquad
 \langle(\mathrm{X}-\mathrm{Y})e,e\rangle\le2\varphi''(s)=-2c_\nu s^{\nu-2},
\end{equation}
and
\begin{equation}\label{eq:FL-XYnorm}
 \|\mathrm{X}\|+\|\mathrm{Y}\|\le \mathrm{C}\left(s^{-1}+s^{\nu-2}\right).
\end{equation}

\smallskip
\noindent\emph{Step 4: contradiction.}
Let
\[
 A_i:=|\xi_i|^\gamma.
\]
In \eqref{eq:rigorous-subtraction}, the localization matrices and
the quadratic time localization contribute at most $\mathrm{C}L_1$ because the total gradients lie in the fixed annulus. The term
$\frac{\zeta}{(T-\bar t)^2}$ on the left is positive and can be discarded.
Thus \eqref{eq:rigorous-subtraction} gives
\begin{equation}\label{eq:FL-visc-subtract}
 0\le
 A_1\bigl(F(\mathrm{X})-F(\mathrm{Y})\bigr)
 +(A_1-A_2)F(\mathrm{Y})
 +\mathrm{C}L_1+2\|f\|_\infty.
\end{equation}
Since $\mathrm{N}:=\mathrm{X}-\mathrm{Y} \le 0$, the upper Pucci inequality gives
\[
 F(\mathrm{X})-F(\mathrm{Y})
 \le\mathscr P^+_{1,1+\delta}(\mathrm{N})
 =\tr \mathrm{N}.
\]
Because $\mathrm{N} \le0$,
$\tr \mathrm{N} \le\langle \mathrm{N}e,e\rangle$, and hence
\begin{equation}\label{eq:FL-main-negative}
 F(\mathrm{X})-F(\mathrm{Y})
 \le-\mathrm{c} s^{\nu-2}.
\end{equation}
On the annulus \eqref{eq:FL-fixed-annulus}, $A_1$ is bounded above and below
by positive universal constants when $\gamma$ is small.  Thus the first term
in \eqref{eq:FL-visc-subtract} is bounded above by
$-\mathrm{c} s^{\nu-2}$.

For the mismatch term, the mean-value theorem applied to $r\mapsto r^\gamma$
on the interval $[\frac{3}{4},K^*+\frac{3}{2}]$, together with $\big||\xi_1|-|\xi_2|\big|\le|\xi_1-\xi_2|$, yields
\[
 |A_1-A_2|
 \le \mathrm{C}_{K^*}|\gamma|\,|\xi_1-\xi_2|.
\]
Also $F(0)=0$ and uniform ellipticity imply
$|F(\mathrm{Y})|\le \mathrm{C}_{n,K^*}\|\mathrm{Y}\|$.  Using
\eqref{eq:FL-xi-diff} and \eqref{eq:FL-XYnorm},
\begin{align*}
 |(A_1-A_2)F(\mathrm{Y})|
 &\le \mathrm{C}_{K^*}|\gamma|Hs^{\frac{\bar\theta}{2}}
 \bigl(s^{-1}+s^{\nu-2}\bigr)\\
 &\le \mathrm{C}_{K^*}|\gamma|Hs^{\nu-2}.
\end{align*}
The last inequality follows because
\[
 1+\frac{\bar\theta}{2}-\nu
 =\frac{\bar\theta}{4}>0.
\]
We first choose $\varepsilon_*$ so that the coefficient of this error is at
most one quarter of the negative coefficient in
\eqref{eq:FL-main-negative}, and then choose $\eta_*$ so small that
$\mathrm{C} L_1+2\eta_*$ is at most another quarter.  Since $s\le2$ and
$\nu<2$, the negative term has a fixed positive lower magnitude.  This
contradicts \eqref{eq:FL-visc-subtract}.

Hence \eqref{eq:FL-unit-lip} holds for $t_0<0$, and by continuity
also for $t_0=0$. Every spatial test gradient has norm at most one.
The bounds $2\le|q|\le K^*$ then give \eqref{eq:FL-annulus}.
\end{proof}

\subsection{Uniformly parabolic estimates}\label{sec:smooth}
We next prove a perturbative estimate for uniformly parabolic equations.
It will apply to the stopping equation after extending its operator
outside the annulus of test gradients.

\begin{lemma}[Small perturbations of the heat equation]\label{lem:FL-smooth}
Let $0<\beta<1$.  There exists $\eta_s=\eta_s(n,\beta)>0$ such that the
following holds.  Let $\mathcal F(p, \mathrm{M})$ be continuous and globally uniformly
parabolic with fixed ellipticity constants $ \frac{1}{2}$ and $3$, with $\mathcal F(p,0)=0$, and assume
\begin{equation}\label{eq:FL-G-close}
 |\mathcal F(p,\mathrm{M})-\tr \mathrm{M}|
 \le\eta_s\|\mathrm{M}\|
 \quad(p\in\mathbb R^n,\ \mathrm{M} \in \text{Sym}(n)).
\end{equation}
If $w$ is a continuous viscosity solution of
\[
 w_t-\mathcal F(Dw,D^2w)= f
 \quad\hbox{in }Q_1,
 \qquad
 \|w\|_{L^\infty(Q_1)}\le1,
 \qquad
 \|f\|_\infty\le\eta_s,
\]
then
\begin{equation}\label{eq:FL-smooth-est}
 \|w\|_{C^{1+\beta,\frac{1+\beta}{2}}(Q_{\frac{1}{2}})}
 \le \mathrm{C}(n,\beta).
\end{equation}
\end{lemma}

\begin{proof}
\emph{Step 1: One-step approximation.} We first choose $\rho_s\in(0,\frac{1}{4})$
so small that $ \mathrm{C}_n\rho_s^2\le\frac{\rho_s^{1+\beta}}{4}$, where $\mathrm{C}_n$ is the interior caloric Taylor constant for functions bounded by one.
There is $\eta_s>0$ such that each normalized solution in the statement
admits an affine function $\ell=a+b\cdot x$ with
\begin{equation}\label{eq:smooth-one-step}
 \sup_{Q_{\rho_s}}|w-\ell|\le\rho_s^{1+\beta},
 \quad |a|+|b|\le \mathrm{C}_n.
\end{equation}
Indeed, otherwise there are solutions $w_j$, operators $\mathcal F_j$,
and sources $f_j$ for which the perturbation constants and $\|f_j\|_\infty $
tend to zero and the conclusion fails. The fixed ellipticity constants
and $\mathcal F_j(p,0)=0$ place $w_j$ in a common parabolic Pucci class.
The interior H\"older estimate \cite{W92a} supplies a locally
uniformly convergent subsequence. At every test contact,
\[
 |\mathcal F_j(D\phi,D^2\phi)-\Delta\phi|
 \le\eta_j\|D^2\phi\|\longrightarrow0.
\]
Viscosity stability shows that the limit $w_\infty$ is caloric. Its
spatial Taylor polynomial at the origin has error at most $C_n\rho_s^2$
on $Q_{\rho_s}$ and bounded coefficients. Local uniform convergence
then contradicts the failure of \eqref{eq:smooth-one-step}.

\emph{Step 2: Iteration.} Set $\ell_0=0$. Suppose
\[
 \sup_{Q_{\rho_s^j}}|w-\ell_j|\le\rho_s^{j(1+\beta)},\qquad
 \ell_j(x)=a_j+b_j\cdot x.
\]
Define
\[
 W_j(x,t):=\frac{w(\rho_s^jx,\rho_s^{2j}t)-\ell_j(\rho_s^jx)}{\rho_s^{j(1+\beta)}}.
\]
Then $\|W_j\|_\infty\le1$, and its operator and source are
\begin{align*}
 \mathcal F_j(p,\mathrm{M})&:=\rho_s^{j(1-\beta)}
 \mathcal F(b_j+\rho_s^{j\beta}p,\rho_s^{j(\beta-1)}\mathrm{M}),\\
 f_j(x,t)&:=\rho_s^{j(1-\beta)}f(\rho_s^jx,\rho_s^{2j}t).
\end{align*}
The ellipticity constants are unchanged, $\mathcal F_j(p,0)=0 $, and
\[
 |\mathcal F_j(p, \mathrm{M})-\tr \mathrm{M}|\le\eta_s\|\mathrm{M}\|,
 \quad \|f_j\|_\infty\le\eta_s.
\]
These estimates are uniform in $b_j$ because the perturbation bound holds
for every gradient. Apply \eqref{eq:smooth-one-step} to $W_j$, obtaining
$A_j+B_j\cdot x$, and set
\[
 a_{j+1}:=a_j+\rho_s^{j(1+\beta)}A_j,\qquad
 b_{j+1}:=b_j+\rho_s^{j\beta}B_j.
\]
This yields the next flatness estimate and
\[
 |a_{j+1}-a_j|\le \mathrm{C}\rho_s^{j(1+\beta)},\qquad
 |b_{j+1}-b_j|\le  \mathrm{C} \rho_s^{j\beta}.
\]
Hence $a_j\to w(0,0)$ and $b_j\to b_\infty$, with geometric tail
bounds. For $\rho_s^{j+1}<r\le\rho_s^j$, the flatness estimate and
the tails give
\[
 \sup_{Q_r}|w-w(0,0)-b_\infty\cdot x|\le  \mathrm{C} r^{1+\beta}.
\]
For $z_0=(x_0,t_0)\in Q_{\frac{1}{2}}$, set
$a=\frac{1}{4}$ and $V(y,s):=w(x_0+ay,t_0+a^2s)$. Then $V$ is defined
on $Q_1$, has supremum norm at most one, and solves an equation with
operator
\[
 \widetilde{\mathcal F}(p, \mathrm{M})=a^2\mathcal F(a^{-1}p,a^{-2}\mathrm{M})
\]
and source $a^2f(x_0+ay,t_0+a^2s)$. The ellipticity constants
and the perturbation bound are unchanged. Repeating the iteration for
$V$ therefore gives the same affine expansion at every center of
$Q_{\frac{1}{2}}$, with uniform constants at radii at most $a$.

\emph{Step 3: Comparison of affine expansions.} At a fixed time, expansions
at two nearby points can be compared on a ball of radius comparable to
their distance; testing the difference of the two affine functions in
coordinate directions gives $|Dw(x,t)-Dw(y,t)|\le \mathrm{C}|x-y|^\beta$.
For $s<t$, put $r=2\sqrt{t-s}$. The expansion based at $(x,t)$
reaches the earlier time $s$. Compare it there with the expansion based
at $(x,s)$ on $B_r(x)$ to obtain
$|Dw(x,t)-Dw(x,s)|\le \mathrm{C}r^\beta$. Evaluating the later expansion at
$x$ itself gives $|w(x,t)-w(x,s)|\le \mathrm{C} r^{1+\beta}$. The uniform slope bounds and boundedness handle larger distances. These estimates prove \eqref{eq:FL-smooth-est}.
\end{proof}

\subsection{Application at the stopping scale}
\begin{lemma}[Reduction of the stopping equation]\label{lem:FL-extension}
At the first stopping scale, if
$|\gamma|+\delta+\|f\|_\infty$ and the normalization constant $\eta$ are
sufficiently small, the normalized deviation $w_k$ satisfies the conclusion
of Lemma~\ref{lem:FL-smooth} in $Q_{\frac{3}{8}}$.  In particular, for
$\sigma=\frac{1}{4}$ there exists $p_k=Dw_k(0,0)$ with $|p_k|\le1$ and
\begin{equation}\label{eq:FL-stop-expansion}
 |w_k(x,t)-w_k(0,0)-p_k\cdot x|
 \le \mathrm{C} \bigl(|x|^{1+\beta}+|t|^{\frac{1+\beta}{2}}\bigr)
 \quad\text{in }Q_\sigma.
\end{equation}
All constants are independent of the stopping level and of the sign of
$\gamma$.
\end{lemma}

\begin{proof}
At the first stopping scale,
\eqref{eq:FL-q-upper} and Lemma~\ref{lem:FL-noncritical} give
\[
 1\le|q_k+p|\le K^*+1
\]
for every viscosity test gradient $p$ of $w_k$ in $Q_{\frac{3}{4}}$.  On this fixed
annulus,
\begin{equation}\label{eq:FL-power-close}
 \sup_{1\le r\le K^*+1}|r^\gamma-1|
 \le \mathrm{C}_{K^*}|\gamma|.
\end{equation}
Let $F_k$ be the rescaled operator.  By Lemma~\ref{Section2.2:lemma1},
\[
 |F_k(\mathrm{M})-\tr \mathrm{M}|
 \le \mathrm{C}_n\delta\|\mathrm{M}\|.
\]
Thus, whenever $1\le|q_k+p|\le K^*+1$,
\begin{align*}
 \bigl||q_k+p|^\gamma F_k(\mathrm{M})-\tr \mathrm{M} \bigr|
 &\le
 \bigl||q_k+p|^\gamma-1\bigr|\,|F_k(\mathrm{M})|
 +|F_k(\mathrm{M})-\tr \mathrm{M}|\\
 &\le \mathrm{C}(|\gamma|+\delta)\|\mathrm{M}\|.
\end{align*}
Define the clamped radius
\[
 R_k(p):=\max\{1,\min\{|q_k+p|,K^*+1\}\},\qquad a_*(p):=R_k(p)^\gamma.
\]
At every test gradient, $a_*(p)=|q_k+p|^\gamma$. For all $p$,
the bounds $1\le R_k(p)\le K^*+1$ give
\[
 \frac{3}{4}\le a_*(p)\le\frac{3}{2},\qquad |a_*(p)-1|\le \mathrm{C}_{K^*}|\gamma|
\]
when $|\gamma|$ is sufficiently small.
Define
\[
 \mathcal G_k(p, \mathrm{M}):=a_*(p)F_k(\mathrm{M}).
\]
Then $\mathcal G_k$ is globally uniformly parabolic and
\[
 |\mathcal G_k(p, \mathrm{M})-\tr \mathrm{M}|
 \le \mathrm{C}(|\gamma|+\delta)\|\mathrm{M}\|.
\]
Because the two operators agree at every viscosity contact gradient of
$w_k$, the function $w_k$ is also a viscosity solution of the extended
equation. Choose the parameters so that
$\mathrm{C}(|\gamma|+\delta)\le\eta_s$ and
$\varepsilon_0\le\eta_s$. Then
$\|f_k\|_{L^\infty(Q_1)}\le\varepsilon_0\le\eta_s$ by
\eqref{eq:FL-fk} and \eqref{eq:FL-beta-gamma}, while
$\|w_k\|_{L^\infty(Q_1)}\le\eta\le1$.
For sufficiently small $\delta$, the ellipticity bounds for
$\mathcal G_k$ lie between $\frac{1}{2}$ and $3$.
A fixed standard parabolic rescaling of $Q_{\frac{3}{4}}$ to $Q_1$
therefore allows us to apply Lemma~\ref{lem:FL-smooth} and obtain
the asserted norm bound on $Q_{\frac{3}{8}}$.  Its first-order Taylor estimate at the origin gives \eqref{eq:FL-stop-expansion}.  The gradient bound $|p_k|\le1$ follows from Lemma~\ref{lem:FL-noncritical}.
\end{proof}

\section{Proof of Theorem~\ref{Main:thm}}\label{Section 6}

In this section, we complete the proof of Theorem~\ref{Main:thm}. We first combine the critical and noncritical estimates from Section~\ref{Section 5} into a common affine expansion for normalized solutions, with constants independent of the stopping scale. After transferring this expansion to the original solution, we compare the affine approximations at nearby points to obtain the required spatial
and temporal H\"older estimates.

\subsection{A common affine expansion}
The critical and noncritical alternatives use different cylinders. We
first place their conclusions on the common cylinders
$\mathcal C_r$ defined in \eqref{eq:FL-comparison-cylinder}.

\begin{proposition}\label{prop:FL-pointwise}
There exist $\eta,\varepsilon_0>0$ and $\mathrm{C}>0$, depending
only on $n$ and $\alpha$, with the following property. Let $F$ be
$(1,1+\delta)$-elliptic with $F(0)=0$, and let $u$ be a continuous
viscosity solution of
\[
 u_t-|Du|^\gamma F(D^2u)=f\quad\text{in }Q_1,
\]
where $\gamma>-1$ and $f\in C^0(Q_1)\cap L^\infty(Q_1)$.
Assume that $u$ satisfies the normalization and smallness conditions
in \eqref{eq:FL-normalized}.
Then a vector $ p\in\mathbb R^n$ exists such that $|p|\le \mathrm{C} $ and
\begin{equation}\label{eq:FL-common-expansion}
 \sup_{\mathcal C_r(0,0)}
 |u(x,t)-u(0,0)-p\cdot x|\le \mathrm{C} r^{1+\beta},
 \qquad 0<r\le1.
\end{equation}
In particular, the spatial derivative at the origin exists and equals
$p$.
\end{proposition}

\begin{proof}
We specify the order of the choices. Fix $\beta$, $\gamma_{\rm FL}$
and $\rho$ as in Section~\ref{Section 4}. This fixes
$K^*=\rho^{-\beta}(3+ \mathrm{C}_0)$ before the oscillation parameter is chosen.
Let $\eta_*,\varepsilon_*$ be the constants in
Lemma~\ref{lem:FL-noncritical} for this annulus, and choose
$0<\eta\le\min\{1,\frac{\eta_*}{2}\}$. Finally choose $\varepsilon_0>0$
below the one-step threshold for $M=3$ and $\eta_0=\eta$, below
$\varepsilon_*$ and $\frac{\eta_*}{2}$, and below the smallness threshold
required by Lemma~\ref{lem:FL-extension}. We also require
$\varepsilon_0\le\min\{\gamma_{\rm FL},\eta_s\}$. Every choice depends only on
$n$ and $\alpha$.

Suppose first that every level is critical. Take $p=0$. For
$0<r\le1$, choose $j\ge0$ such that
$r\le r_j<\frac{r}{\rho}$, with $j=0$ at $r=1$. We have
$\mathcal C_r\subset Q_{r_j}^{\mu_j}$. Evaluating the flatness estimate
at the origin gives
\[
 |a_j-u(0,0)|\le\eta r_j^{1+\beta}.
\]
Since $|l_j|\le3r_j^\beta$, it follows that
\[
 \sup_{\mathcal C_r}|u-u(0,0)|
 \le2\eta r_j^{1+\beta}+3r_j^\beta r
 \le \mathrm{C} r^{1+\beta}.
\]

Now suppose $k\ge1$ is the first noncritical level. Write
$r_0=r_k$ and $T_0=r_0^{2-\beta\gamma}$. By
Lemma~\ref{lem:FL-extension}, $p_k=Dw_k(0,0)$ exists and $|p_k|\le1$.
Set
\[
 p=l_k+r_0^\beta p_k,
 \quad |p|\le(K^*+1)r_0^\beta.
\]
Scaling \eqref{eq:FL-stop-expansion} back to $u$ gives
\begin{equation}\label{eq:FL-stop-unscaled}
 |u(x,t)-u(0,0)-p\cdot x|
 \le \mathrm{C} \left(|x|^{1+\beta}
 +r_0^{\frac{\beta\gamma(1+\beta)}{2}}|t|^{\frac{1+\beta}{2}}\right)
\end{equation}
whenever $|x|\le\sigma r_0$ and $-\sigma^2T_0\le t\le0$.
Put $\vartheta=\vartheta_\gamma=\max\{2,2-\beta\gamma\}$.
For $r\le\sigma r_0$, the cylinder $\mathcal C_r$ lies in this
Taylor cylinder. Indeed, if $\gamma\ge0$ then
$\frac{r^2}{T_0}\le(\frac{r}{r_0})^2$, whereas if $\gamma<0$ then
$\frac{r^\vartheta}{T_0}=(\frac{r}{r_0})^\vartheta\le(\frac{r}{r_0})^2$.
For $\gamma\ge0$, the time term in
\eqref{eq:FL-stop-unscaled} is at most $ \mathrm{C} r^{1+\beta}$ because
$r_0^{\frac{\beta\gamma(1+\beta)}{2}}\le1$. For $\gamma<0$, it is at most
\[
 \mathrm{C}r^{1+\beta}
 \left(\frac{r}{r_0}\right)^{-\frac{\beta\gamma(1+\beta)}{2}}
 \le \mathrm{C}r^{1+\beta}.
\]
Thus \eqref{eq:FL-common-expansion} holds below $\sigma r_0$.

If $\sigma r_0\le r\le r_0$, use the level-$k$ flatness estimate,
which remains valid at the first failed critical level. Since
$\mathcal C_r\subset Q_{r_0}^{\mu_k}$,
\[
 \sup_{\mathcal C_r}|u-u(0,0)-p\cdot x|
 \le2\eta r_0^{1+\beta}+(|l_k|+|p|)r
 \le \mathrm{C} r^{1+\beta}.
\]
For $r\ge r_0$, choose $j\le k$ with $r\le r_j<\frac{r}{\rho}$.
The slope at this level satisfies $|l_j|\le \mathrm{C} r_j^\beta$: this is the critical bound for $j<k$ and \eqref{eq:FL-q-upper} for $j=k$. Also $|p|\le \mathrm{C}r_0^\beta\le \mathrm{C}r_j^\beta$ and
$\mathcal C_r\subset Q_{r_j}^{\mu_j}$. Therefore
\[
 \sup_{\mathcal C_r}|u-u(0,0)-p\cdot x|
 \le2\eta r_j^{1+\beta}+(|l_j|+|p|)r
 \le \mathrm{C}r^{1+\beta}.
\]
This proves \eqref{eq:FL-common-expansion} in both alternatives.
Restricting to $t=0$ and letting $r \to 0$ identifies $p$ with
the spatial derivative.
\end{proof}

\subsection{Normalization and comparison of the affine expansions}\label{subsec:FL-normalization-comparison}
\begin{proof}[Proof of Theorem~\ref{Main:thm}]
Let $\eta,\varepsilon_0$ be the constants in
Proposition~\ref{prop:FL-pointwise}. Decrease $\varepsilon_0$ if
necessary so that $\varepsilon_0\le\frac{1}{4}$, and take the structural
threshold in the theorem to be $\varepsilon=\frac{\varepsilon_0}{2}$.
Write
\[
 U:=\|u\|_{L^\infty(Q_1)},\qquad H:=\|f\|_{L^\infty(Q_1)},
\]
and choose the fixed normalization constants
\begin{equation}\label{eq:FL-data-normalization}
 A:=\max\left\{1,\frac{2U}{\eta},\frac{2H}{\varepsilon_0}\right\},
 \qquad
 R:=\min\left\{\frac{1}{8},
 \left(\frac{\lambda}{8A^{\frac{1}{4}}}\right)^{\frac{4}{7}}\right\},
 \qquad
 T:=\frac{R^{2+\gamma}}{\lambda A^\gamma}.
\end{equation}
These choices work for either sign of $\gamma$. In fact,
$|\gamma|\le\frac{1}{4}$, $R\le1$ and $A\ge1$ imply
\begin{equation}\label{eq:FL-time-normalization-bounds}
 \frac{R^{\frac{9}{4}}}{\lambda A^{\frac{1}{4}}}
 \le T\le\frac{R^{\frac{7}{4}}A^{\frac{1}{4}}}{\lambda}\le\frac{1}{8}.
\end{equation}

Fix $z_0=(x_0,t_0)\in B_{\frac{1}{2}}\times(-\frac{1}{4},0)$ and define
\[
 v(y,s):=\frac{u(x_0+Ry,t_0+Ts)-u(z_0)}{A}.
\]
The physical cylinder used here lies inside $Q_1$, because
$R\le\frac{1}{8}$ and $T\le\frac{1}{8}$. A direct change of variables gives
\[
 v_s-|Dv|^\gamma G(D^2v)=h\quad\hbox{in }Q_1,
\]
where
\[
 G(\mathrm{M}):=\frac{R^2}{\lambda A}F\left(\frac{A}{R^2}\mathrm{M}\right),
 \qquad
 h(y,s):=\frac{T}{A} f(x_0+Ry,t_0+Ts).
\]
The operator $G$ vanishes at zero and has ellipticity constants
$1$ and $1+\delta$, where $\delta=\frac{\Lambda}{\lambda}-1$. Moreover,
\[
 v(0,0)=0,\qquad \|v\|_{L^\infty(Q_1)}\le\frac{2U}{A}\le\eta,
 \qquad
 \|h\|_\infty\le\frac{H}{8A}\le\frac{\varepsilon_0}{16}.
\]
Also $\osc v\le\frac{\osc_{Q_1}u}{A}\le\frac{2U}{A}\le\eta$. Thus
$|\gamma|+\delta+\|h\|_\infty\le\varepsilon_0$, and
Proposition~\ref{prop:FL-pointwise} applies.

We now express this estimate in the original coordinates. Keep
$\vartheta=\max\{2,2-\beta\gamma\}$ and set
\[
 K=\max\{R^{-1},T^{-\frac{1}{\vartheta}}\}.
\]
The physical cylinder $\mathcal C_r(z_0)$ is mapped into the normalized
cylinder $\mathcal C_{Kr}(0,0)$. The spatial inclusion follows from
$\frac{r}{R}\le Kr$, and the time inclusion from
$\frac{r^\vartheta}{T}\le(Kr)^\vartheta$. Consequently, the spatial derivative
at $z_0$ exists, $|Du(z_0)|\le \frac{\mathrm{C} A}{R}$, and
\begin{equation}\label{eq:FL-physical-expansion}
 \sup_{\mathcal C_r(z_0)}
 |u(x,t)-u(z_0)-Du(z_0)\cdot(x-x_0)|
 \le \mathrm{C} A K^{1+\beta}r^{1+\beta},  \quad 0<r\le K^{-1}.
\end{equation}
All constants can be made independent of $\gamma$. Indeed,
$2\le\vartheta\le2+\frac{\beta}{4}<\frac{9}{4}$, and
\eqref{eq:FL-time-normalization-bounds} gives
\[
 K\le\overline K
 :=\max\{R^{-1},\lambda^{\frac{1}{2}}A^{\frac{1}{8}}R^{-\frac{9}{8}}\}.
\]
Here we used $T<1$ and $\vartheta\ge2$ to obtain
$T^{-\frac{1}{\vartheta}}\le T^{-\frac{1}{2}}$. In particular,
\eqref{eq:FL-physical-expansion} holds at every such center, for the
common radius range $0<r\le r_*:=(4\overline K)^{-1}$, with common
right-hand constant $ \mathrm{C}_*= \mathrm{C}A\overline K^{1+\beta}$.

For clarity, we give the affine comparison in detail. At a fixed time
$t<0$, let $x,y\in B_{\frac{1}{2}}$ and $d=|x-y|\le \frac{r_*}{4}$. Set $r=4d$.
Let $P_x$ and $P_y$ be the affine functions supplied by
\eqref{eq:FL-physical-expansion}, centered at $(x,t)$ and $(y,t)$.
The points $x$ and $x+(\frac{r}{2})e_i$ lie in both spatial balls of radius
$r$. At time $t$, each of $P_x$ and $P_y$ differs from $u$ at these
points by at most $ \mathrm{C}_*r^{1+\beta}$. Hence
\[
 \left|(Du(x,t)-Du(y,t))\cdot\frac{r}{2}e_i\right|
 \le4 \mathrm{C}_*r^{1+\beta}.
\]
Summing the coordinate estimates yields
\begin{equation}\label{eq:FL-gradient-space}
 |Du(x,t)-Du(y,t)|\le \mathrm{C}|x-y|^\beta.
\end{equation}

Next fix $x\in B_{\frac{1}{2}}$ and $-\frac{1}{4}<s<t<0$. Put
$\tau=t-s$ and $r=(2\tau)^{\frac{1}{\vartheta}}$. If $r\le r_*$, then
the expansion centered at $(x,t)$ is valid on the slice at time $s$.
Compare it there with the expansion centered at $(x,s)$, first at
$x$ and then at $x+(\frac{r}{2})e_i$. Subtracting these two comparisons
cancels the constant terms and gives
\begin{equation}\label{eq:FL-gradient-time}
 |Du(x,t)-Du(x,s)|\le \mathrm{C}r^\beta
 \le \mathrm{C}|t-s|^{\frac{\beta}{\vartheta}}.
\end{equation}
Evaluating the later expansion at $(x,s)$ itself also gives
\begin{equation}\label{eq:FL-value-time}
 |u(x,t)-u(x,s)|\le \mathrm{C} r^{1+\beta}
 \le \mathrm{C}|t-s|^{\frac{1+\beta}{\vartheta}}.
\end{equation}
The small-time condition holds whenever
$\tau\le \frac{r_*^{\frac{9}{4}}}{2}$, uniformly in $\gamma$.
For larger distances in \eqref{eq:FL-gradient-space} or
\eqref{eq:FL-gradient-time}, use $|Du|\le \frac{\mathrm{C}A}{R}$. For larger time differences in \eqref{eq:FL-value-time}, use $|u|\le U$.
Thus all three inequalities hold throughout the indicated domain,
with a constant depending only on the data in the theorem.

These estimates also justify the top time $t=0$. For each $x$, the
uniform time modulus in \eqref{eq:FL-gradient-time} gives a limit
$p_0(x)=\lim_{t\uparrow0}Du(x,t)$. The convergence is uniform in $x$,
and \eqref{eq:FL-gradient-space} passes to this limit. At a fixed small
radius, take $t \to 0$ in the spatial part of
\eqref{eq:FL-physical-expansion}. Continuity of $u$ gives
\[
 |u(y,0)-u(x,0)-p_0(x)\cdot(y-x)|\le \mathrm{C}_*|y-x|^{1+\beta}
\]
after increasing the constant if necessary. Hence $Du(x,0)$ exists and
equals $p_0(x)$. The spatial, temporal and value estimates all extend
to $t=0$ by continuity.

Finally, \eqref{eq:FL-time-conversion} and the definition of
$\vartheta$ imply
\[
 \beta>\alpha,   \quad
 \frac{\beta}{\vartheta}\ge\frac{\alpha}{2}, \quad
 \frac{1+\beta}{\vartheta}\ge\frac{1+\alpha}{2}.
\]
The spatial gradient estimate, the temporal gradient estimate and the
temporal estimate for $u$ therefore give precisely
\[
 \|u\|_{C^{1+\alpha,\frac{1+\alpha}{2}}(Q_{\frac{1}{2}})}\le \mathrm{C}.
\]
The dependence of $A$, $R$ and $\overline K$ in
\eqref{eq:FL-data-normalization} shows that $ \mathrm{C} $ depends only on $n$, $\alpha$, $\lambda$, $U$ and $H$, as claimed.
\end{proof}

\end{document}